\documentclass[11pt]{amsart}

\usepackage{parskip}
\allowdisplaybreaks

\usepackage{geometry} 
\usepackage{framed} 

\usepackage{tikz}
\tikzset{every picture/.style={line width=1pt}} 

\usepackage{setspace}
\setdisplayskipstretch{2.5}

\usepackage{amsthm} 
\usepackage{amsmath}
\usepackage{amssymb}

\usepackage{hyperref}

\usepackage{amsfonts} 
\newcommand\A{\mathcal{A}}
\newcommand\B{\mathcal{B}}
\newcommand\FB{\mathfrak{B}}
\newcommand\C{\mathbb{C}}
\newcommand\CC{\mathcal{C}}

\newcommand\BC{\mathbf{C}}

\newcommand\G{\mathcal{G}}
\renewcommand\H{\mathcal{H}}

\newcommand\FL{\mathfrak{L}}
\newcommand\M{\mathcal{M}}
\newcommand\CN{\mathcal{N}}

\newcommand\CR{\mathcal{R}}
\newcommand\R{\mathbb{R}}
\renewcommand\S{\mathcal{S}}
\newcommand\Z{\mathbb{Z}}

\newcommand\w{\omega}
\newcommand\vphi{\varphi}
\renewcommand\phi{\vphi}

\usepackage{stmaryrd} 

\newcommand\id{\textnormal{id}}
\newcommand\sing{\textnormal{sing}}
\newcommand\reg{\textnormal{reg}}
\newcommand\spt{\textnormal{spt}}

\newcommand\dist{\textnormal{dist}}
\newcommand\graph{\textnormal{graph}}
\newcommand\ext{\mathrm{d}}
\newcommand\del{\partial}
\newcommand\vartan{\textnormal{VarTan}}

\newcommand{\res}{\mathbin{\hspace{0.1em}\vrule height 1.3ex depth 0pt width 0.13ex\vrule height 0.13ex depth 0pt width 1.0ex}} 

\newcommand{\weakly}{\rightharpoonup}
\renewcommand{\div}{\textnormal{div}}

\makeatletter
\def\@tocline#1#2#3#4#5#6#7{\relax
  \ifnum #1>\c@tocdepth 
  \else
    \par \addpenalty\@secpenalty\addvspace{#2}%
    \begingroup \hyphenpenalty\@M
    \@ifempty{#4}{%
      \@tempdima\csname r@tocindent\number#1\endcsname\relax
    }{%
      \@tempdima#4\relax
    }%
    \parindent\z@ \leftskip#3\relax \advance\leftskip\@tempdima\relax
    \rightskip\@pnumwidth plus4em \parfillskip-\@pnumwidth
    #5\leavevmode\hskip-\@tempdima
      \ifcase #1
       \or\or \hskip 1em \or \hskip 2em \else \hskip 3em \fi%
      #6\nobreak\relax
    \dotfill\hbox to\@pnumwidth{\@tocpagenum{#7}}\par
    \nobreak
    \endgroup
  \fi}
\makeatother

\newtheoremstyle{newtheoremstyle}
{3pt}
{3pt}
{\itshape}
{\parindent}
{\bfseries}
{.}
{0.5em}
{} 

\newtheoremstyle{newtheoremstyledefn}
{3pt}
{3pt}
{}
{\parindent}
{\bfseries}
{.}
{0.5em}
{}

\theoremstyle{newtheoremstyle}
\newtheorem{theorem}{Theorem}
\newtheorem*{theorem*}{Theorem}
\newtheorem{lemma}[theorem]{Lemma}

\newtheorem{thmx}{Theorem}

\theoremstyle{newtheoremstyledefn}
\newtheorem{defn}[theorem]{Definition}

\numberwithin{equation}{section} 
\numberwithin{theorem}{section}

\usepackage{fancyhdr}
\begin{document}

\title{The Structure of Stable Codimension One Integral Varifolds near Classical Cones of Density $Q+1/2$}

\author{
	Paul Minter
}

\address{\textnormal{Princeton University (Fine Hall) and the Institute for Advanced Study (Fuld Hall), Princeton, New Jersey, USA, 08544}}
\email{pm6978@princeton.edu, pminter@ias.edu}

\begin{abstract}
For each positive integer $Q\in\Z_{\geq 2}$, we prove a multi-valued $C^{1,\alpha}$ regularity theorem for varifolds in the class $\S_Q$, i.e., stable codimension one stationary integral $n$-varifolds which have no classical singularities of vertex density $<Q$, which are sufficiently close to a stationary integral cone comprised of $2Q+1$ half-hyperplanes (counted with multiplicity) meeting along a common axis. Such a result furthers the understanding of the local structure about singularities in the (possibly branched) varifolds in $\S_Q$ achieved by the author and N.~Wickramasekera (\cite{minterwick}) and generalises the authors' previous work in the case $Q=2$ (\cite{minter-5-2}) to arbitrary $Q\in \Z_{\geq 2}$. One notable difference with previous works is that our methods do not need any a priori size restriction on the (density $Q$) branch set to rule out density gaps.
\end{abstract} 

\maketitle

\tableofcontents

\section{Introduction and Main Result}\label{sec:intro}

The understanding of branch points within stationary integral varifolds -- that is, singular points where at least one tangent cone is a plane of multiplicity $>1$ and on no neighbourhood of the singular point is the varifold a sum of smooth minimal submanifolds -- is one of the most elusive open problems in the regularity theory of the area functional. Even in the setting of area-minimising integral currents, it is still not known whether the tangent cone at a density $2$ branch point is unique. One is therefore naturally led to the problem of developing techniques which can be used to understand branch points arising within various classes of stationary integral varifolds. The first significant step in this direction was due to F.~Almgren (\cite{almgren}; see also \cite{DLSp1, DLSp2, DLSp3}), who introduced the frequency function (as well as the centre manifold) to prove that the branch set in an area-minimiser has codimension at least 2 within the support of the varifold; this bound is sharp (as illustrated by simple examples of complex analytic varieties such as $\{(z,w)\in \C\times\C: z^2=w^3\}$), however, no information regarding uniqueness of the tangent plane, let alone the local structure of the area-minimiser about a branch point, is provided. It should be noted that in codimension one, area-minimisers do not exhibit branch points due to their superposition as Caccioppoli sets (see \cite{simon-gmt}).

Recently, another large class of stationary integral varifolds has enjoyed significant developments in its regularity theory: this is the class of codimension one stationary integral $n$-varifolds $V$ (where $n\geq 2$ and, say, in $B^{n+1}_2(0)\subset\R^{n+1}$) which have stable regular part (in the sense of $(\S2)$ below). Originally, the work of R.~Schoen and L.~Simon (\cite{SS}) showed that for such varifolds if one assumes that the size of the singular set is sufficiently small, then branch points do not occur (and indeed a certain sheeting theorem holds whenever such a varifold is close in Hausdorff distance to a hyperplane) and the singular set has dimension at most $n-7$. This result was later extended by N.~Wickramasekera (\cite{wickstable}) under the significantly weaker assumption that the $\H^{n-1}$-measure of the singular set vanishes or, equivalently (as shown in \cite{wickstable}), that the varifold does not contain any so-called \textit{classical singularities}, i.e., singular points $X\in \spt\|V\|$ such that there is some $\rho>0$ such that $V\res B_\rho(X) = \sum_{i}|M_i|$, where this is a finite sum of (at least 3 distinct) $C^{1,\alpha}$ submanifolds-with-boundary $M_i$, for some $\alpha\in (0,1)$ such that $B_\rho(X)\cap\del M_i = B_\rho(X)\cap \del M_j$ for each $i,j$ and each pair $M_i,M_j$ is either everywhere disjoint (away for their common boundary) or $M_i = M_j$; we remark that in fact by the work of B.~Krummel (\cite{krummel-reg}), these $M_i$ can be taken to be $C^\infty$ submanifolds-with-boundary (or in fact $C^\omega$ as we are in $\R^{n+1}$ with the usual metric). In particular, Wickramasekera's regularity theory shows that any branch point in a stable codimension one stationary integral varifold must be a limit of classical singularities. As codimension one area-minimisers have stable regular part and do not exhibit classical singularities (which is readily seen by simple 1-dimensional comparison arguments), Wickramasekera's work demonstrates that the regularity theory of codimension one area-minimisers is actually a direct consequence of the regularity theory for codimension one stable integral varifolds.

Aside from Almgren's dimension bound for area-minimisers, all the results mentioned above in fact a posteriori \textit{rule out} branch points. Examples of branched stable codimension one stationary integral varifolds have been constructed (see \cite{simonwick-examples}, \cite{krummel-examples}, and \cite{rosales}), and thus given the results of \cite{wickstable}, this is a natural class to study in order to develop techniques for the understanding of branch points. One must therefore try to prove regularity results in the presence of classical singularities and branch points. Notable examples of rectifiability results for branch sets in this direction include the work of Krummel--Wickramasekera for two-valued stationary graphs (\cite{KW3}) and multi-valued functions locally minimising the Dirichlet energy (\cite{KW2}; see also \cite{DMSV}).

Recently, in \cite{minterwick} (which strengthens and generalises earlier work seen in \cite{wick08} and \cite{wickstable}), the author and Wickramasekera studied the class $\S_Q$ (here, $Q\in\Z_{\geq 2}$) of stable codimension one stationary integral varifolds which do not contain any classical singularities of (vertex) density $<Q$ (i.e. no classical singularities formed of $<2Q$ smooth submanifolds-with-boundary, counted with multiplicity) and showed that a multi-valued sheeting theorem holds locally about branch points of density $Q$; in particular, the tangent cone is unique at such points. More precisely, it was shown (see \cite[Theorem A]{minterwick} or Theorem \ref{thm:MW1} below) that if $V\in \S_Q$ is close, as varifolds, to a hyperplane of multiplicity $Q$, then $V$ is locally expressible as a Lipschitz graph $u$ which is in fact \textit{generalised}-$C^{1,\alpha}$ (for some $\alpha = \alpha(n,Q)$); here, the ``generalised'' refers to the fact that classical singularities (of density $Q$) need not immersed for $Q\geq 3$, yet nonetheless the individual sheets are separately $C^{1,\alpha}$ submanifolds-with-boundary (see Definition \ref{defn:GC1a} for a precise definition). 

Among the techniques used in the work \cite{minterwick} is a novel method for establishing monotonicity of Almgren's frequency function associated with (multi-valued) blow-ups of varifolds in $\S_{Q}$; the frequency function is then employed, as its primary use, to establish a fundamental regularity estimate for the blow-ups rather than for proving, as in the area minimising setting of \cite{almgren}, a dimension bound on the branch set (although such a bound can still be achieved). In contrast to the area-minimising setting, it is not clear a priori that blow-ups of varifolds in $\S_{Q}$ satisfy any variational property. In the area minimising case, the blow-ups are a priori shown to be locally Dirichlet energy minimising, which, via first variation arguments, directly leads to frequency monotonicity. The monotonicity of frequency in \cite{minterwick} must be achieved by other means, and indeed the argument proceeds by combining variational identities for the varifolds in $\S_{Q}$ about classical singularities with a new, elementary energy non-concentration estimate at flat singular points of the blow-ups.

Moreover, analogously to the area-minimising setting, it can be readily shown that (the varifold associated to) a codimension one area-minimising current mod $p$ belongs to the class $\S_{p/2}$, and thus such a result is applicable to understanding multiplicity $p/2$ branch points in mod $p$ minimisers (see \cite[Theorem E]{minterwick}). It should be noted that when $p$ is odd, from the work of B.~White (\cite{white-mod-p}) it was already known that branch points do not occur, and when $p$ is even the only branch points which can occur necessarily have density $p/2$ (and in fact one must have $p\geq 6$ from \cite{white-mod-4}). However, this work heavily relies on the area-minimising mod $p$ assumption.

The next natural question is whether it is possible to understand density $>Q$ branch points for varifolds in the class $\S_Q$. An intermediate step towards this question is to first understand singular points to $V\in \S_Q$ which have a tangent cone consisting of half-hyperplanes meeting along a common axis; we remark that by \cite{wickstable}, such a tangent cone necessarily has $\geq 2Q$ half-hyperplanes (always counted with multiplicity), and when the number of half-hyperplanes is exactly $2Q$ it follows that the singular point is in fact a classical singularity of density $Q$ (see \cite[Theorem C]{minterwick} or Theorem \ref{thm:MW3}). Thus, the next case of interest is when the tangent cone consists of $2Q+1$ half-hyperplanes meeting along a common boundary; such a situation is the topic of study for the current work. We remark that the situation of $\geq 2Q+2$ half-hyperplanes (and thus density $\geq Q+1$ branch points) is significantly different to the setting of $2Q+1$ half-hyperplanes due to the possibility of \textit{density gaps}, i.e., $V\in \S_Q$ can be close to a union of $2Q+2$ half-hyperplanes with common boundary but not itself have a $\H^{n-1}$-positive measure set of points of density $Q+1$ (we show, in Lemma \ref{lemma:gaps}, that for the class $\S_Q$ such density gaps do not occur in the case of $2Q+1$ half-hyperplanes with a common boundary). Dealing with density gaps when they occur is a significant obstacle for our current techniques (a notable example where this issue is present yet dealt with is \cite{beckerkahn}).

For the class $\S_2$, i.e. when $Q=2$, the corresponding regularity question at singular points having a tangent cone comprised of $5$ half-hyperplanes meeting along a common axis has been answered in \cite{minter-5-2}; however, compared with the current work, \cite{minter-5-2} uses significantly more structural information for $V\in \S_2$, namely that: (i) the dimension of the multiplicity 2 branch set is at most $n-2$; (ii) that density $2$ classical singularities are immersed; and (iii) that blow-ups of varifolds relative to a multiplicity two hyperplane are $C^{1,1/2}$ two-valued harmonic functions (see \cite[Theorem D]{minterwick} and also \cite{simonwick-frequency}).

In the present work, we will overcome these difficulties and see that in fact one does not need this extra information to establish a regularity result for $V\in \S_Q$ near singular points which have a tangent cone comprised of $2Q+1$ half-hyperplanes with a common boundary. We stress that we do \textit{not} need to assume a dimension bound on the (density $Q$) branch set of $V\in \S_Q$ for our argument. However, the iterative blow-up procedures seen in \cite{minter-5-2} are still necessary for our argument, although they are essentially the same as in \cite{minter-5-2} (once certain other properties have been established) and so we will not discuss them here. The main purpose of the present work is therefore to focus on proving these other properties in the general setting. One of these is to rule out density gaps, and the other is to prove a boundary regularity theorem for certain blow-up classes; for this latter point, we give an argument which allows one in certain circumstances to extend monotonicity of a frequency function at \textit{interior} points to monotonicity of a frequency function at \textit{boundary} points, which is then used to establish the boundary regularity of functions in the blow-up class.

Let us now recall the definition of the class $\S_Q$, as in \cite{minterwick}, and state our main result.

{\bf Definition:} Let $Q \in \{2,3,\dotsc\}.$ Let $\S_{Q}$ denote that class of integral $n$-varifolds $V$ on the open ball $B^{n+1}_2(0) \subset \R^{n+1}$ with $0 \in \spt\|V\|$, $\|V\|(B_{2}^{n+1}(0)) < \infty$ and which satisfy the following conditions: 
\begin{enumerate}
	\item [$(\S1)$\;\;\,] $V$ is stationary in $B^{n+1}_2(0)$ with respect to the area functional, in the following (usual) sense: for any given vector field $\psi\in C^1_c(B^{n+1}_2(0);\R^{n+1})$, $\epsilon>0$, and $C^2$ map $\phi:(-\epsilon,\epsilon)\times B^{n+1}_2(0)\to B^{n+1}_2(0)$ such that:
	\begin{enumerate}
		\item [(i)] $\phi(t,\cdot):B^{n+1}_2(0)\to B^{n+1}_2(0)$ is a $C^2$ diffeomorphism for each $t\in (-\epsilon,\epsilon)$ with $\phi(0,\cdot)$ equal to the identity map on $B^{n+1}_2(0)$,
		\item [(ii)] $\phi(t,x) = x$ for each $(t,x)\in (-\epsilon,\epsilon)\times\left(B^{n+1}_2(0)\setminus\spt(\psi)\right)$, and
		\item [(iii)] $\left.\del\phi(t,\cdot)/\del t\right|_{t=0} = \psi$,
	\end{enumerate}
	we have that
	$$\left.\frac{\ext }{\ext t}\right|_{t=0} \|\phi(t,\cdot)_\# V\|(B^{n+1}_2(0)) = 0;$$
	equivalently (see \cite[Section 39]{simon-gmt}),
	$$\int_{B^{n+1}_2(0)\times G_n}\div_S\psi(X)\ \ext V(X,S) = 0$$
	for every vector field $\psi\in C^1_c(B^{n+1}_2(0);\R^{n+1})$.
	\item [$(\S2)$\;\;\,] $\reg(V)$ is stable in $B^{n+1}_2(0)$ in the following (usual) sense: for each open ball $\Omega\subset B^{n+1}_2(0)$ with $\sing(V)\cap \Omega = \emptyset$ in the case $2\leq n\leq 6$ or $\H^{n-7+\gamma}(\sing(V)\cap \Omega) = 0$ for every $\gamma>0$ in the case $n\geq 7$, given any vector field $\psi\in C^1_c(\Omega\setminus \sing(V);\R^{n+1})$ with $\psi(X)\perp T_X\reg(V)$ for each $X\in \reg(V)\cap \Omega$,
	$$\left.\frac{\ext^2}{\ext t^2}\right|_{t=0}\|\phi(t,\cdot)_\#V\|(B^{n+1}_2(0))\geq 0$$
	where $\phi(t,\cdot)$, $t\in (-\epsilon,\epsilon)$, are the $C^2$ diffeomorphisms of $B^{n+1}_2(0)$ associated with $\psi$, described in $(\S1)$ above; equivalently (see \cite[Section 9]{simon-gmt})\footnote{ This equivalence requires two-sidedness of ${\rm reg} \, V$, which holds in a ball $\Omega$ as above in view of the smallness assumption on the singular set in $\Omega.$}  for every such $\Omega$ we have
	$$\int_{\reg\,V\cap \Omega}|A|^2\zeta^2\ \ext\H^n \leq \int_{\reg\,V\cap \Omega}|\nabla\zeta|^2\ \ext\H^n\ \ \ \ \text{for all }\zeta\in C^1_c(\reg(V)\cap\Omega)$$
	where $A$ denotes the second fundamental form of $\reg(V)$, $|A|$ the length of $A$, and $\nabla$ the gradient operator on $\reg(V)$.
	\item[$(\S3)_Q$] No singular point of $V$ with density $< Q$ is a classical singularity of $V$ (see Definition \ref{defn:classical-singularity}).
\end{enumerate}

Let $\BC = \sum^N_{i=1}q_i|H_i|$ be a stationary classical cone in $\R^{n+1}$ with density $\Theta_{\BC}(0) = Q+1/2$ and spine $S(\BC) = \{0\}^2\times\R^{n-1}$, i.e., $q_i\geq 1$ are integers and $H_i$ are distinct half-hyperplanes with $\del H_i = S(\BC)$ for each $i=1,2,\dotsc,N$, such that $\sum_{i}q_i = 2Q+1$ (in particular, as $N\geq 3$ we have $q_i\geq Q$ for at most two $i$) and, if we write $H_i = R_i\times\R^{n-1}$, where $R_i:= \{t\mathbf{w}_i:t>0\}$, for distinct unit vectors $\mathbf{w}_1,\dotsc,\mathbf{w}_N\in \R^2$, we have $\sum_i q_i\mathbf{w}_i = 0$; in particular, $\BC = \BC_0\times\R^{n-1}$, where $\BC_0 = \sum_{i=1}^Nq_i|R_i|$. Now let $\sigma_0:= \max_{i\neq k}\{\mathbf{w}_i\cdot\mathbf{w}_k\}$, and let $N(H_i)$ denote the conical neighbourhood of $H_i$ defined by:
$$N(H_i):= \left\{(x,y)\in \R^2\times\R^{n-1}:x\cdot\mathbf{w}_i>|x|\sqrt{\frac{1+\sigma_0}{2}}\right\}.$$
Denote by $\tilde{H}_i$ the hyperplane containing $H_i$, and by $\tilde{H}_i^\perp$ the orthogonal complement of $\tilde{H}_i$ in $\R^{n+1}$. Also, write $r(X) = |x|$, where $X = (x,y)\in \R^2\times\R^{n-1}$.

Our main result is then the following:

\begin{thmx}\label{thm:A}
	Let $\BC$ be as above. Then, there is a constant $\epsilon = \epsilon(\BC_0,n)$ such that the following is true: if $V\in \S_Q$ has $\Theta_V(0)\geq \Theta_{\BC}(0)$ and
	$$\int_{B_1^{n+1}(0)}\dist^2(X,\spt\|\BC\|)\ \ext\|V\|<\epsilon,\ \ \ \ \text{and}$$
	$$\|V\|(B^{n+1}_{1/2}(0)\setminus\{r(X)<1/8\}\cap N(H_i)) \geq \left(q_i - 1/4\right)\H^n(B^{n+1}_{1/2}(0)\setminus\{r(X)<1/8\}\cap H_i)$$
	for each $i\in \{1,\dotsc,N\}$, then, for each $i\in \{1,\dotsc,N\}$, there is a function
	$$\gamma_i\in C^{1,\alpha}\left(\overline{S(\BC)\cap B^{n+1}_{1/2}(0)};\tilde{H}_i\cap \{X:r(X)<1/16\}\right)$$
	and functions $u_{i}:\Omega_i\to \A_{q_i}(\tilde{H}^\perp_i$), where $\Omega_i$ is the connected component of $\tilde{H}_i\cap B^{n+1}_{1/2}(0)\setminus\{x+\gamma_i(x):x\in S(\BC)\cap B_{1/2}^{n+1}(0)\}$ with $\left(H_i\setminus\{r(X)<1/16\}\right)\cap B^{n+1}_{1/2}(0)\subset\Omega_i$, such that:
	 \begin{enumerate}
	 	\item [(i)] $u_i\in GC^{1,\alpha}(\overline{\Omega}_i;\A_{q_i}(\tilde{H}_i))$ (see Definition \ref{defn:GC1a}) with $\mathbf{v}(u_i)$ a stationary integral varifold, where $\mathbf{v}(u)$ is the varifold $(\graph(u_i),\theta)$, where if we write $u_i(X) = \sum^{q_i}_{j=1}\llbracket u_i^j(X)\rrbracket$ for $X\in \Omega_i$, then $\graph(u_i):= \{(u_i^j(X),X): X\in\Omega_i,\; j=1,\dotsc,q_i\}$, and the multiplicity function $\theta$ at a point $(u^j_i(X),X)\in \graph(u_i)$ is given by $\theta(u_i^j(X),X):= \#\{k:u^k_i(X) = u^j_i(X)\}$ for each $j=1,\dotsc,q_i$;
		\item [(ii)] for each $i=1,\dotsc,N$, $\left. u_i\right|_{\del \Omega_i\cap B^{n+1}_{1/2}(0)} = q_i\llbracket b_i\rrbracket$, where $b_i:\del \Omega_i\cap B^{n+1}_{1/2}(0)\to \tilde{H}^\perp_i$ is a single-valued $C^{1,\alpha}$ function, and moreover if $\tilde{b}_i(x):= x+b_i(x)$ for $x\in \del\Omega_i\cap B^{n+1}_{1/2}(0)$, then $\textnormal{image}(\tilde{b}_i) = \textnormal{image}(\tilde{b}_j)$ for all $i,j$;
		\item [(iii)] $V\res B^{n+1}_{1/2}(0) = \sum^{N}_{i=1}\mathbf{v}(u_i)\res B^{n+1}_{1/2}(0);$
		\item [(iv)] for each $i=1,\dotsc,N$,
		$$\{Z:\Theta_V(Z)\geq Q+1/2\}\cap B^{n+1}_{1/2}(0) = \{Z:\Theta_V(Z) = Q+1/2\} = \tilde{b}_i(\del \Omega_i\cap B^{n+1}_{1/2}(0));$$
		 \end{enumerate}
		here, $\alpha = \alpha(n,Q)\in (0,1)$. Moreover, for each $i=1,\dotsc,N$, we have
		$$|u_i|_{1,\alpha;\overline{\Omega}_i}\leq C\left(\int_{B_1^{n+1}(0)}\dist^2(X,\spt\|\BC\|)\ \ext\|V\|(X)\right)^{1/2}$$
		where $C = C(n)\in (0,\infty)$. In particular, $V$ has a unique tangent cone at every point in $B^{n+1}_{1/2}(0)$, and $\{X:\Theta_V(X) = Q+1/2\}\cap B^{n+1}_{1/2}(0)$ is a connected $C^{1,\alpha}$ submanifold of dimension $n-1$.
\end{thmx}

\textbf{Remark:} More can be said in the situation where $V$ has a singular point with a tangent cone taking this form. In this case, Theorem \ref{thm:A} holds on a neighbourhood of the point, and moreover as the varifold has a tangent cone, it tells us that all the functions over a given half-hyperplane must have a derivative at a point $\del\Omega_i\cap B^{n+1}_{1/2}$ which agrees with that of the corresponding half-hyperplane in the tangent cone. As such, over half-hyperplanes of density $q_i<Q$, the function $u_i$ is in fact a single $C^{1,\alpha}$ function of multiplicity $q_i$; this follows from the fact that the sheets of $u_i$ are ordered (i.e. we can write $u_i = \sum_{\alpha=1}^{q_i}\llbracket u^\alpha_i\mathbf{w}_i^\perp\rrbracket$, with $u_i^1\leq \cdots \leq u_i^{q_i}$ and $u_i^\alpha\in C^{1,\alpha}$ for each $i$; here $\mathbf{w}_i^\perp$ is the orthogonal vector to $\mathbf{w}_i$ in $\R^2$ extended to $\R^{n+1}$), which follows from \cite{wickstable} (see also Theorem \ref{thm:MW2}) and the Hopf boundary point lemma. Of course, in the general situation this need not be true, as there is nothing to guarantee agreement of the derivatives at a boundary point for the ordered sheets, and it is indeed entirely possible for a higher multiplicity half-hyperplane to split into several nearby lower multiplicity half-hyperplanes.\hfill $\blacktriangle$

We then have the following corollary of Theorem \ref{thm:A} regarding the local structure of varifolds $V\in \S_Q$ in the (open) region $\{\Theta_V<Q+1\}$:

\begin{thmx}\label{thm:B}
	Let $V\in \S_Q$. Then
	$$\spt\|V\|\cap B^{n+1}_{1}(0)\cap \{\Theta_V<Q+1\} = \Omega\cup \B\cup \CC_Q\cup \CC_{Q+1/2}\cup K$$
	where:
	\begin{enumerate}
		\item [(i)] $\Omega$ is the set of points $X\in \spt\|V\|\cap B^{n+1}_1(0)\cap \{\Theta_V<Q+1\}$ such that for some $\delta_X >0$, $\spt\|V\|\cap B_{\delta_X}(X)$ is a smoothly embedded hypersurface;
		\item [(ii)] \textnormal{(\cite[Theorem A]{minterwick})} $\B$ is the set of points $X\in \spt\|V\|\setminus\Omega$ where one tangent cone to $V$ at $X$ is of the form $Q|P|$ for some hyperplane $P$; moreover, this is the unique tangent cone to $V$ at $X$, and there is a $\delta_X>0$ such that $V\res B_{\delta_X}(X)$ is given by a $GC^{1,\alpha}$ $Q$-valued function over a domain in $P$, where $\alpha = \alpha(n,Q)$;
		\item [(iii)] \textnormal{(\cite[Theorem C]{minterwick})} $\CC_Q$ is the set of points $X\in \spt\|V\|$ such that one tangent cone to $V$ at $X$ is a classical cone of vertex density $Q$; in particular, $X$ is a classical singularity of $V$;
		\item [(iv)] $\CC_{Q+1/2}$ is the set of points $X\in \spt\|V\|$ such that one tangent cone to $V$ at $X$ is a classical cone of vertex density $Q+1/2$; moreover, this is the unique tangent cone to $V$ at $X$, and the conclusions of Theorem \ref{thm:A} hold in some neighbourhood of $X$;
		\item [(v)] $K = \S_{n-2}$ is the usual $(n-2)$-stratum of the singular set of $V$; in particular, it is countably $(n-2)$-rectifiable (\cite{naber-valtorta}).
	\end{enumerate}
\end{thmx}

\textbf{Remark:} It is possible that, in general, in the region $\{\Theta_V\leq Q+1/2\}$ we have $\S_{n-2}\setminus\S_{n-3} = \emptyset$, where $\S_{n-3}$ is the $(n-3)$-stratum of $\sing(V)$; if this were true, in the region $\{\Theta_V\leq Q+1/2\}$ we could take $K$ to be $\S_{n-3}$, which is countably $(n-3)$-rectifiable. This is because points $X\in \S_{n-2}\setminus\S_{n-3}$ have a tangent cone of the form $\BC =\BC_0\times\R^{n-2}$, where $\BC_0\subset\R^3$ is a stationary integral cone in $\R^3$. As such the link, $\Sigma:= \BC_0\res S^2$, is a stationary integral $1$-varifold in the unit sphere $S^2\subset\R^3$, and so is a sum of great circle arcs meeting at junction points. However, in the present setting where $V$ has no classical singularities of density $<Q$, and in the region where $\Theta_V<Q+1/2$, every junction must be formed of $2Q$(possibly non-distinct) great circle arcs meeting at a point. When $Q=1$, it follows immediately that $\BC_0$ must be a single great circle (or multiplicity one). When $Q=2$, then as all junctions are formed of $4$-junction points, then necessarily $\BC_0$ is a sum of great circles, and so this forces $\Theta_{\BC_0}(0)\geq 3$; so in particular that in the region $\{\Theta_V\leq 5/2\}$ we could take $K$ to be $\S_{n-3}$ when $Q=2$. It is entirely possible that the other configurations, and for general $Q$, require $\Theta_{\BC_0}(0)\geq Q+1/2$. In the region $\{Q+1/2<\Theta_V <Q+1\}$, it is possible to have points in $\S_{n-2}$ (for example, when $Q=1$, the tetrahedral cone has density between $3/2$ and $2$ -- see \cite{polyhedral}).

We note that our results have natural generalisations to the setting where the ambient space is an arbitrary Riemannian manifold $(M^{n+1},g)$, and the changes to the proofs are only technical (in a similar manner to that seen in \cite[Section 18]{wickstable}).

\subsection{Overview of Paper}

We shall only detail the aspects of the proof which are different to that seen in \cite{minter-5-2}; as such, we will not detail any blow-up, fine blow-up, or ultra fine blow-up procedures, as they will be identical to those seen in \cite{minter-5-2}. Our main focus will therefore be on proving that density gaps do not occur for our class (see Section \ref{sec:density-gaps}) and that one can still prove a $GC^{1,\alpha}$ boundary regularity statement for the various blow-up classes (see Section \ref{sec:blow-up}); this latter fact will require an extension of suitable variational identities which hold in the interior of a half-ball to boundary points, which is shown in Section \ref{sec:squash-squeeze}. In Section \ref{sec:prelim} we will provide all the preliminary material needed from \cite{minterwick} for our results, including the various varifold regularity results as well as properties of (coarse) blow-ups of sequences of varifolds in $\S_Q$ converging to a multiplicity $Q$ hyperplane. In Section \ref{sec:classes}, we define the various classes of varifolds we are interested in, and introduce the notion of a \textit{multiplicity $Q$ class} analogous to that of a multiplicity two class seen in \cite{minter-5-2} (which, just as the notion of multiplicity one class used in \cite{simoncylindrical}, is a natural notion in the setting when the varifold is close to a classical cone of half-integer vertex density). The rest of the paper will then detail the necessary modifications to the argument seen in \cite{minter-5-2} needed to establish Theorem \ref{thm:A}.

\textbf{Acknowledgements:} The author would like to thank Neshan Wickramasekera for numerous helpful discussions over the course of this work.

\section{Preliminaries}\label{sec:prelim}

In this section we recall our basic set-up and various notions and results from the works \cite{minter-5-2} and \cite{minterwick} which we shall use in the present paper.

\subsection{Basic Notation and Definitions}

We work in $\R^{n+1}$ throughout. Often we will work with coordinates $X = (x,y)\in \R^2\times\R^{n-1}\cong\R^{n+1}$, where the $x$-coordinates will be those for the cross-section of a cone and the $y$-coordinates for its spine; with these coordinates, we write $r(X):= |x|$ and $R(X):= |X|$. We also write, for $\rho>0$ and $X\in\R^{n+1}$, $B^{n+1}_\rho(X):= \{Y\in \R^{n+1}:|Y-X|<\rho\}$, however often we will suppress the superscript and just write $B^{n+1}_\rho(X)\equiv B_\rho(X)$ and only write the superscript when we wish to stress the dimension of the (open) ball. Much of our language will be the same as in \cite{simon-gmt}, although more specialised language can be found in \cite{minter-5-2} and \cite{minterwick}; we only recall here the notions which are less standard so to assist the reader.

For $V$ a stationary integral $n$-varifold in $B^{n+1}_{2}(0)$, we write $\reg(V)$ for its \textit{regular part}, which is the set of points $X\in \spt\|V\|$ such that there is a $\rho>0$ for which $\spt\|V\|\cap B_\rho(X)$ is a smoothly embedded $n$-submanifold in $B_\rho(X)$. The (\textit{interior}) \text{singular set} of $V$ is then $\sing(V):= (\spt\|V\|\setminus \reg(V))\cap B^{n+1}_2(0)$, and for $X\in \sing(V)$, we write $\vartan_X(V)$ for the set of tangent cones to $V$ at $X$. For $\BC\in \vartan_X(V)$, we have $\Theta_\BC(0) = \Theta_V(X)$, and we write
$$S(\BC):= \{Y\in \R^{n+1}:\Theta_{\BC}(Y) = \Theta_{\BC}(0)\}$$
for the \textit{spine} of $\BC$; it is standard the $S(\BC)$ is equal to the set of points along which $\BC$ is translation invariant, and thus is a subspace of $\R^{n+1}$. Hence, one can find a rotation $q:\R^{n+1}\to \R^{n+1}$ such that $q_\#\BC = \BC_0\times\R^{\dim(S(\BC))}$, where $\BC_0$ is a $(n-\dim\,S(\BC))$-dimensional stationary cone in $\R^{n+1-\dim(S(\BC))}$. In general, a \textit{classical cone} in $\R^{n+1}$ is an integral varifold $\BC$ of the type $\BC = \sum^N_{j=1}q_j|H_j|$ where $N$ is an integer $\geq 3$, $H_1,\dotsc,H_N$ are half-hyperplanes with a common boundary (also called the \textit{spine}) $S(\BC)=\del H_j$ for all $j=1,\dotsc,N$, and $q_j$ are positive integers; we write $\CC_Q$ for the set of classical cones of vertex density $Q$ in $\R^{n+1}$. Hence, $\BC\in \vartan_X(V)$ is a classical cone when $\dim(S(\BC)) = n-1$, i.e., up to rotation, $\BC = \BC_0\times \R^{n-1}$, where now $\BC_0\subset \R^2$ is a stationary integral cone, and hence we can write $\BC = \sum_{i=1}^Nq_i|H_i|$, where $N$ is an integer $\geq 3$ and $H_i$ are half-hyperplanes with $\del H_i =\del H_j$ ($=S(\BC)$) for each $i,j$. If $\BC$ is a classical cone, it readily follows that $\Theta_\BC(0) \in \{3/2,2,5/2,\dotsc\}$.

\textbf{Definition.} Let $V$ be a stationary integral $n$-varifold in $B^{n+1}_2(0)$. We say that $X\in \sing(V)$ is a \textit{branch point} if at least one tangent cone to $V$ at $X$ is supported on a hyperplane, yet there is no neighbourhood of $X$ on which $\spt\|V\|$ is a union of finitely many (smoothly) embedded submanifolds. Write $\B_V\equiv\B$ for the set of branch point singularities of $V$.

The usual stratification of the singular set, $\S_0\subset\S_1\subset\cdots\subset\S_{n-1}\subset\sing(V)$, is given by $\S_{j}:= \{X\in\sing(V):\dim(S(\BC))\leq j\text{ for all }\BC\in\vartan_X(V)\}$.

\begin{defn}\label{defn:classical-singularity}
Let $T\subset\R^{n+1}$ be a subset. A $C^1$ (resp. $C^{1,\alpha}$, for $\alpha\in (0,1)$) \textit{classical singularity} of $T$ is a point $y\in T$ such that for some $\rho>0$ and some integer $N\geq 3$, $T \cap B_\rho^{n+1}(y) = \cup_{j=1}^N M_j$, where $M_j\subset B^{n+1}_\rho(y)$ are embedded $C^1$ (resp. $C^{1,\alpha}$) submanifolds-with-boundary having the same $(n-1)$-dimensional $C^1$ boundary $L = \del M_j$ for each $j=1,2,\dotsc,N$, with $y\in L$, and $M_j\cap M_i = L$ for $i\neq j$, and with $M_i$ and $M_j$ intersecting transversely at every point of $L$ for at least one pair of indices $i,j$.
\end{defn}

We say that $y\in T$ is a \textit{classical singularity} if $y$ is a $C^{1,\alpha}$ classical singularity of $T$ for some $\alpha\in (0,1)$. For $V$ an $n$-varifold on $\R^{n+1}$, $y\in \spt\|V\|$ is a $C^1$ \textit{classical singularity} (resp. $C^{1,\alpha}$ \textit{classical singularity}, \textit{classical singularity}) of $V$ if it is a $C^1$ classical singularity (resp. $C^{1,\alpha}$ classical singularity, classical singularity) of $\spt\|V\|$.

For $A\subset\R^n$ and a $Q$-valued function $f:A\to \A_{Q}(\R)$, we denote by $\CC_f$ the set of points $x\in A$ such that $f^1(x) = f^2(x) = \cdots = f^Q(x)$ and the point $(f^1(x),x)$ is a $C^1$ classical singularity of $\graph(f)$.

We now recall the notion of \textit{generalised}-$C^1$ and \textit{generalised}-$C^{1,\alpha}$ regularity for an $\A_Q(\R)$-valued function $f$, as defined in \cite{minterwick}:

\begin{defn}\label{defn:GC1}
	Let $U\subset\R^n$ be an open set. We say that a function $f:U\to \A_Q(\R)$ belongs to $GC^1(U)$, or equivalently that $f$ is \textit{generalised}-$C^1$ in $U$, if:
	\begin{enumerate}
		\item [(i)] $f$ is differentiable (as a function on $U$) at every point $y\in U\setminus\CC_f$ in the classical sense, i.e. taking $f$ to be an $\R^Q$-valued function $x\mapsto (f^1(x),\dotsc,f^Q(x))$ (with $f^1\leq \cdots\leq f^Q$ always), and;
		\item [(ii)] the derivative $Df$ is continuous on $U\setminus \CC_f$.
	\end{enumerate}
\end{defn}

\begin{defn}\label{defn:branch-regular-sets}
	Let $U\subset\R^n$ be an open set and let $f\in GC^1(U)$. We write $\CR_f$ for the set of points $y\in U$ for which there is $\rho_y \in (0,\dist(y,\del U))$ such that $\left. f^1\right|_{B_{\rho_y}(y)},\dotsc,\left. f^Q\right|_{B_{\rho_y}(y)}\in C^1(B_{\rho_y}(y);\R)$; we call $\CR_f$ the \textit{regular set} of $f$.
	
	We set $\B_f:= U\setminus (\CR_f\cup \CC_f)$ and call $\B_f$ the \textit{branch set} of $f$.
\end{defn}

\begin{defn}\label{defn:GC1a}
	Let $U\subset \R^n$ be open and $\alpha\in (0,1)$. We say that a function $f:U\to \A_Q(\R)$ is \textit{generalised}-$C^{1,\alpha}$, or equivalently $f\in GC^{1,\alpha}(U)$, if $f\in GC^1(U)$ and, with the notation as in Definition \ref{defn:branch-regular-sets}: (a) for each compact $K\subset U$, each $y\in \R_f$ and for the largest $\rho_y$ corresponding to $y$, the functions $f^j$, $j=1,\dotsc,Q$, are in $C^{1,\alpha}(\overline{B}_{\rho_y}(y)\cap K)$; (b) each $y\in \CC_f$ is a $C^{1,\alpha}$ classical singularity of $\graph(f)$; (c) the map $Df$ is in $C^{0,\alpha}_{\text{loc}}(\CR_f\cup\B_f;\A_Q(\mathcal{M}_{1\times n}))$, i.e. for each compact set $K\subset \CR_f\cup \B_f$,
	$$\sup_{x_1,x_2\in K:\;x_1\neq x_2}\frac{\G(Df(x_1),Df(x_2))}{|x_1-x_2|^\alpha}<\infty.$$
\end{defn}

\subsection{Known Regularity Results for the Class $\S_Q$}

Throughout the current work we will need various regularity results from \cite{minterwick} for varifolds $V\in \S_Q$ (as defined in Section \ref{sec:intro}); we record the statements here to aid the reader.

\begin{theorem}[Theorem A, \cite{minterwick}]\label{thm:MW1}
	There is a number $\epsilon = \epsilon(n,Q)\in (0,1)$ such that if $V\in \S_Q$, $(\w_n2^n)^{-1}\|V\|(B^{n+1}_2(0))<Q+1/2$, $Q-1/2\leq \w_n^{-1}\|V\|(\R\times B^n_1(0))<Q+1/2$, and $\int_{\R\times B^n_1(0)}|x^1|^2\ \ext\|V\|<\epsilon$, then we have the following: there is a generalised-$C^{1,\alpha}$ ($Q$-valued) function $u:B^{n}_{1/2}(0)\to \A_Q(\R)$ such that:
	\begin{enumerate}
		\item [(i)] $\CR_u = \pi(\{X\in \spt\|V\|\cap (\R\times B^n_{1/2}(0)):\Theta_V(X)<Q\})$; $\B_u\cup\CC_u = \pi(\{X\in\R\times B^n_{1/2}(0): \Theta_V(X)\geq Q\})$, where $\pi:\R^{n+1}\to \{0\}\times\R^n$ is the orthogonal projection, and $\{X\in \R\times B^n_{1/2}(0):\Theta_V(X)>Q\} = \emptyset$;
		\item [(ii)] $$\sup_{B_{1/2}^n(0)}|u| + \sup_{B^n_{1/2}(0)\setminus\CC_u}|Du| \leq C\left(\int_{\R\times B^n_1(0)}|x^1|^2\ \ext\|V\|(X)\right)^{1/2};$$
		\item [(iii)] $V\res (\R\times B^n_{1/2}(0)) = \mathbf{v}(u)\res (\R\times B^n_{1/2}(0))$.
	\end{enumerate}
	In particular, every singular point $Y$ of $V\res (\R\times B^n_{1/2}(0))$ is either a density $Q$ classical singularity or a density $Q$ branch point, with a unique tangent cone $\CC_Y$ at $Y$ in either case with $\pi^{-1}(\pi(Y)) = \{Y\}$, and moreover
	$$\rho^{-n-2}\int_{\R\times B^n_\rho(\pi(Y))}\dist^2(X,\spt\|\BC_Y\|)\ \ext\|V\|(X) \leq C\rho^{2\alpha}\int_{\R\times B_1^n(0)}|x^1|^2\ \ext\|V\|(X)\ \ \ \ \text{for all }\rho\in (0,1/4]$$
	and, for any singular points $Y_1,Y_2$ of $V\res (\R\times B^n_{1/2}(0))$, we have
	$$\dist_\H(\spt\|\BC_{Y_1}\|\cap B_1^{n+1}(0),\spt\|\BC_{Y_2}\|\cap B^{n+1}_1(0)))\leq C|Y_1-Y_2|^\alpha\left(\int_{\R\times B^n_1(0)}|x^1|^2\ \ext\|V\|(X)\right)^{1/2};$$
	here, $\alpha = \alpha(n,Q)\in (0,1)$ and $C = C(n,Q)\in (0,\infty)$.
\end{theorem}

One should note that if we also assume in Theorem \ref{thm:MW1} that there are no classical singularities of density $Q$ in $V\in \S_Q$, then we necessarily have $\{X\in \sing(V)\cap (\R\times B^n_{1/2}(0)): \Theta_V(X)\geq Q\} = \emptyset$, and thus $V\res (\R\times B_{1/2}^n(0))$ is expressible as $Q$ single-valued functions, which are distinct everywhere or coincide; this result is also a consequence of the main sheeting theorem in \cite[Theorem 3.3]{wickstable}, which also holds whenever $V\in \S_Q$ is close, as varifolds, to a hyperplane of multiplicity $<Q$:

\begin{theorem}[Theorem 3.3, \cite{wickstable}]\label{thm:MW2}
	There exists a number $\epsilon_0 = \epsilon_0(n,Q)\in (0,1)$ such that if $V\in \S_Q$ satisfies $(\w_n 2^n)^{-1}\|V\|(B^{n+1}_{2}(0))<Q-1/4$ and $\int_{\R\times B_1}|x^1|^2\ \ext\|V\|<\epsilon_0$, then we have
	$$V\res (\R\times B_{1/2}) = \sum^q_{j=1}|\graph(u_j)|$$
	for some $q\in \{1,2,\dotsc,\lfloor Q-1/2\rfloor\}$, where $\lfloor x\rfloor$ denotes the integer part of $x$, $u_j\in C^{1,\beta}(B_{1/2})$ for each $j$, $u_1\leq u_2\leq \cdots \leq u_q$, and:
	$$|u_j|_{C^{1,\beta}(B_{1/2})}\leq C\left(\int_{\R\times B_1}|x^1|^2\ \ext\|V\|(X)\right)^{1/2}$$
	where $C = C(n,Q)\in (0,\infty)$ and $\beta = \beta(n,Q)\in (0,1)$; furthermore, $u_j$ solves the minimal surface equation weakly on $B_{1/2}$, and hence in fact $u_j\in C^\infty(B_{1/2})$ for each $j\in \{1,\dotsc,q\}$.
\end{theorem}

Theorem \ref{thm:MW1} and Theorem \ref{thm:MW2} both relate to the situation where $V\in \S_Q$ is close to a hyperplane of multiplicity $\leq Q$. We also have a structural result when $V\in \S_Q$ is close to a classical cone of vertex density $Q$. This uses the same notation used in Theorem \ref{thm:A}, except now the cone has density $\Theta_{\BC}(0) = Q$:

\begin{theorem}[Theorem C, \cite{minterwick}]\label{thm:MW3}
	Fix $Q\in\Z_{\geq 2}$ and let $\BC$ be a classical cone with $\Theta_{\BC}(0) = Q$. Let $\alpha\in (0,1)$. Then, using the same notation as in Theorem \ref{thm:A} for $\BC$ (with $Q$ in place of $Q+1/2$), there is a constant $\epsilon = \epsilon(\BC_0,n,\alpha)$ such that the following is true: if $V\in \S_Q$ obeys the assumptions of Theorem \ref{thm:A} with respect to this $\BC$ for this $\epsilon$, then the conclusions of Theorem \ref{thm:A} hold, and moreover for each $i=1,\dotsc,N$, the function $u_i:\Omega_i\to \A_{q_i}(\tilde{H}^\perp_i)$ can be written as $u_i = \sum^{q_i}_{j=1}\llbracket u_{i,j}\rrbracket$, where $u_{i,j}:\Omega_i\to \tilde{H}^\perp_i$ for $j=1,\dotsc, q_i$, obey:
	\begin{enumerate}
		\item [(a)] $u_{i,1}\cdot \nu_i \leq u_{i,2}\cdot\nu_i \leq \cdots\leq u_{i,q_i}\cdot\nu_i$, where $\nu_i$ is a constant unit normal to $\tilde{H}_i$;
		\item [(b)] $u_{i,j}\in C^{1,\alpha}(\overline{\Omega}_i,\tilde{H}^\perp_i)$ and $u_{i,j}\cdot \nu_i$ solves the minimal surface equation on $\Omega_i$;
		\item [(c)] $\left.u_{i,j}\right|_{\del\Omega_i\cap B^{n+1}_{1/2}(0)} = b_i$ for each $j=1,\dotsc,q_i$;
	\end{enumerate}
	moreover, for each $i=1,\dotsc,N$, and $j=1,\dotsc,q_i$,
	$$|u_{i,j}|_{1,\alpha;\Omega_i}\leq C\left(\int_{B^{n+1}_1(0)}\dist^2(X,\spt\|\BC\|)\ \ext\|V\|(X)\right)^{1/2};$$
	here, $C = C(n,Q,\alpha)\in (0,\infty)$; in particular, $V$ has the structure of a classical singularity in $B^{n+1}_{1/2}(0)$.
\end{theorem}

\textbf{Remark:} By \cite{krummel-reg} we in fact know that the functions $u_{i,j}$ are smooth up-to-the-boundary (which is also a smooth boundary).\hfill $\blacktriangle$

\subsection{Blow-Ups of $\S_Q$ Relative to Hyperplanes of Multiplicity at most $Q$}

In this section we detail known results regarding (coarse) blow-ups of sequences of varifolds in $\S_Q$ converging to a (fixed) multiplicity $Q$ disk $Q|\{0\}\times B^n_1(0)|$; all these facts are proved in detail within \cite{minterwick}.

Let $(V_k)_k$ be a sequence of $n$-dimensional stationary integral varifolds on $B^{n+1}_2(0)$ such that for each $k=1,2,3,\dotsc$:
\begin{equation*}\tag{$\star$}
	(\w_n 2^n)^{-1}\|V_k\|(B^{n+1}_2(0))<Q+1/2;\ \ \ \ Q-1/2\leq\w_n^{-1}\|V_k\|(\R\times B^n_1(0)) < Q+1/2.
\end{equation*}
Assume also that $\hat{E}_{V_k}:= \left(\int_{\R\times B^n_1(0)}|x^1|^2\ \ext\|V_k\|(X)\right)^{1/2} \to 0$, where $X = (x^1,x^2,\dotsc,x^{n+1})$. Let $\sigma\in (0,1)$. By applying \cite[Corollary 3.11]{almgren}\footnote{Of course, when we restrict to $(V_k)_k\subset \S_Q$, we can instead use Theorem \ref{thm:MW3}.}, for all sufficiently large $k$, there exist Lipschitz functions $u^j_k:B^n_\sigma(0)\to \R$, $j=1,\dotsc,Q$, with $u^1_k\leq u^2_k\leq \cdots \leq u^Q_k$ and $\text{Lip}(u^j_k)\leq 1/2$ for each $j\in\{1,2,\dotsc,Q\}$ and such that
$$\spt\|V_k\|\cap (\R\times(B_\sigma\setminus\Sigma_k))=\bigcup^Q_{j=1}\graph(u^j_k)\cap (\R\times(B_\sigma\setminus\Sigma_k))$$
where for each $k$, $\Sigma_k\subset B_\sigma$ is a measurable subset with $\H^n(\Sigma_k) + \|V_k\|(\R\times\Sigma_k)\leq C\hat{E}_{V_k}^2$ for some $C = C(n,Q,\sigma)$. Now set $v^j_k(x):= \hat{E}^{-1}_{V_k}u^j_k(x)$ for $x\in B_\sigma$, and write $v_k = (v^1_k,\dotsc,v^Q_k)$. Then $v_k$ is Lipschitz on $B_\sigma$, and moreover it can be readily checked (see \cite[Inequalities (5.8) and (5.9)]{wickstable}) that $\|v_k\|_{W^{1,2}(B_\sigma)}\leq C$ for some $C = C(n,Q,\sigma)$. Thus as $\sigma\in (0,1)$ is arbitrary, we can apply Rellich's compactness theorem and a diagonal argument to obtain a function $v\in W^{1,2}_{\text{loc}}(B_1;\R^Q)\cap L^2(B_1;\R^Q)$ and a subsequence $(k_j)$ of $(k)$ such that $v_{k_j}\to v$ as $j\to\infty$, strongly in $L^2(B_\sigma;\R^Q)$ and weakly in $W^{1,2}(B_\sigma;\R^Q)$ for every $\sigma\in (0,1)$.

\textbf{Definition.} Let $v\in W^{1,2}_{\text{loc}}(B_1;\R^Q)\cap L^2(B_1;\R^Q)$ correspond, in the manner described above, to (a subsequence of) a sequence $(V_k)_k$ of stationary integral $n$-varifolds of $B^{n+1}_2(0)$ satisfying $(\star)$ and with $\hat{E}_{V_k}\to 0$. We call such a $v$ a (\textit{flat}) \textit{coarse blow-up} of the sequence $(V_k)_k$.

\textbf{Definition.} We write $\FB_Q$ for the collection of all coarse blow-ups of sequences of varifolds $(V_k)_k\subset\S_Q$ satisfying $(\star)$ and for which $\hat{E}_{V_k}\to 0$.

From \cite{minterwick}, we know that $\FB_Q$ obeys the following properties:

\begin{theorem}[Section 3, \cite{minterwick}]\label{thm:MW4}
	There exists $\alpha = \alpha(n,Q)\in (0,1)$ such that if $v\in \FB_Q$, then $v\in GC^{1,\alpha}(B_{1}(0);\A_Q(\R))$. Moreover, for each $z\in B_{1/2}$, either \textnormal{(i)} there exists $\rho_z>0$ such that $\Delta v = 0$ in $B_{\rho_z}(z)$ (i.e. each $v^i$ is harmonic in $B_{\rho_z}(z)$), or \textnormal{(ii)} $v^i(z) = v^j(z)$ for all $i,j$ and there exists $\phi_z:\R^n\to \A_Q(\R)$ with $\mathbf{v}(\phi)\in \CC_Q$ or $\mathbf{v}(\phi) = Q|L|$ for some hyperplane $L$ such that for every $0<\sigma\leq\rho/2\leq 3/16$, we have
	$$\sigma^{-n-2}\int_{B_\sigma(z)}\G(v(x)-v_a(z),\phi(x-z))^2\ \ext x\leq C\left(\frac{\sigma}{\rho}\right)^{2\alpha}\cdot\rho^{-n-2}\int_{B_\rho(z)}|v|^2$$
	where $C = C(n,Q)\in (0,\infty)$. Furthermore, we have $v_a$ is harmonic in $B_1$, and:
	\begin{equation}
		\int_{B_1}|Dv|^2\zeta = -\int_{B_1}\sum^Q_{\alpha=1}v^\alpha Dv^\alpha \cdot D\zeta\ \ \ \ \text{for all }\zeta\in C^1_c(B_{3/4};\R);
	\end{equation}
	and
	\begin{equation}
		\int_{B_1}\sum^Q_{\alpha=1}\sum^n_{i,j=1}\left(|Dv^\alpha|^2\delta_{ij} - 2D_i v^\alpha D_j v^\alpha\right) D_i\zeta^j = 0\ \ \ \ \text{for all }\zeta\in C^1_c(B_{3/4};\R^n).
	\end{equation}
	In particular, the frequency function
	$$N_{v,y}(\rho):= \frac{\rho^{2-n}\int_{B_\rho(y)}|Dv|^2}{\rho^{1-n}\int_{\del B_\rho(y)}|v|^2}$$
	is well-defined and non-decreasing whenever $v$ is not identically zero on a neighbourhood of $y$, and thus the frequency $N_v(y):= \lim_{\rho\downarrow 0}N_{v;y}(\rho)\in [0,\infty)$ exists for each $y\in \Omega$. Finally, we have $\dim_\H(\B_v)\leq n-2$.
\end{theorem}

\begin{proof}
	The regularity conclusions and decay estimates follow from \cite[Theorem 3.12]{minterwick}. The fact that $v_a$ is always harmonic follows from \cite[Section 2.6, $(\FB3)$]{minterwick}. The second inequality identity, known as the \textit{squeeze identity}, follows from the $GC^{1,\alpha}$ regularity conclusion and \cite[Lemma 3.8]{minterwick}, whilst the first integral identity, known as the \textit{squash identity}, follows from \cite[Lemma 3.4]{minterwick}; note that in \cite[Lemma 3.4]{minterwick} the identity is in fact an inequality, however using Theorem \ref{thm:MW1}, we are able to show that in fact we have $v_k\to v$ strongly in $W^{1,2}(B_{7/8})$ for the blow-up sequence generating $v$, and thus the argument in \cite[Lemma 3.4]{minterwick} now gives an equality. The claim regarding the frequency function follows from the squash and squeeze identities: see \cite[Theorem 3.9]{minterwick}. The final claim on the size of the singular set follows from \cite[Theorem A.1]{minterwick}.
\end{proof}

\textbf{Remark:} The class $\FB_Q$ also obeys a unique continuation property for elements which are homogeneous of degree one, namely: if $v\in \FB_Q$ is homogeneous of degree one in an annulus $B_1\setminus \overline{B}_r$ for some $r\in (0,1)$, then $v$ is homogeneous of degree one in $B_1$ (see \cite[Lemma 3.2]{minterwick}). \hfill $\blacktriangle$

\section{Classes of Varifolds}\label{sec:classes}

We now describe the set-up for the proof of Theorem \ref{thm:A}. Note that we know $\S_Q$ is a closed class by \cite[Theorem C]{minterwick}.

\begin{defn}
	For $I\in \Z_{\geq 0}$, we say that an integral cone $\BC\in \CC_{Q+1/2}$ is \textit{level $I$} if $\BC$ contains exactly $I$ half-hyperplanes of multiplicity $Q$. We write $\FL_I$ for the set of $\BC\in \CC_{Q+1/2}$ of level $I$ which have $S(\BC) = \{0\}^2\times\R^{n-1}$.\footnote{Note that we may no stationarity assumption on the cones in $\FL_I$.}
\end{defn}

\textbf{Remark:} Let $\BC\in \CC_{Q+1/2}$ be a stationary integral cone; after a suitable rotation, we can write $\BC =\BC_0\times\R^{n-1}$, where $\BC_0\subset\R^2$ is a 1-dimensional stationary integral cone. Let us write $\{n_1,\dotsc,n_k\}$ for the unit vectors in the (outward) directions of the rays of $\spt\|\BC_0\|$ and $\{\theta_1,\dotsc,\theta_k\}\subset\Z_{\geq 1}$ their respective multiplicities in $\BC_0$. The stationarity of $\BC_0$ requires $\sum_i \theta_i n_i = 0$, whilst the density condition requires $\sum_i \theta_i = 2Q+1$; as $\BC$ is not a hyperplane, we necessarily have $k\geq 3$. It follows immediately from these two facts that $\theta_i\in\{1,\dotsc,Q\}$, and that if $\BC$ is of level $I$, then necessarily $I\in \{0,1,2\}$. \hfill $\blacktriangle$

Let us set $\FL:= \FL_0\cup \FL_1\cup \FL_2$; the set $\FL$ will comprise of all cones which are interested in for our proof of Theorem \ref{thm:A}. Write also $\FL_S\subset\FL$ for the set of cones in $\FL$ which are also stationarity as varifolds.

\begin{defn}
	For $V\in \S_Q$ and $\BC\in \FL$, the \textit{one-sided height excess} of $V$ relative to $\BC$ is
	$$E^2_{V,\BC}:= \int_{B_1}\dist^2(X,\spt\|\BC\|)\ \ext\|V\|.$$
\end{defn}

We then introduce a suitable class of ``nearby'' varifolds to each $\BC\in \FL$:

\begin{defn}\label{defn:near}
	For $\BC\in \FL$ and $\epsilon>0$, define $\CN_{\epsilon}(\BC)$ to be the class of $V\in \S_Q$ which obey both:
	\begin{itemize}
		\item $\|V\|(B_1)\in (\|\BC\|(B_1) - 1/8, \|\BC\|(B_1) + 1/8)$;
		\item $E_{V,\BC}<\epsilon$;
		\item For each half-hyperplane $H$ in $\spt\|\BC\|$, if $q_H$ is the multiplicity of $H$ in $\BC$, we have:
		$$\|V\|(B_{1/2}\cap\{|x|>1/8\}\cap N(H))\geq (q_H-1/4)\H^n(B_{1/2}\cap\{|x|>1/8\}\cap H)$$
		where $N(H)$ is the conical neighbourhood of $H$ for the cone $\BC$, as defined in Section \ref{sec:intro}.
	\end{itemize}
\end{defn}

For $\epsilon>0$, we also define the class $\FL_{\epsilon}(\BC)$ of nearby cones to a given cone $\BC\in \FL$ in the following manner: $\BC^\prime\in \FL_{\epsilon}(\BC)$ if $\BC^\prime\in \FL$ and, if we write $\BC = \BC_0\times\R^{n-1}$ and further write $\BC_0 = \sum^{2Q+1}_{i=1}|\ell_i|$ for some (not necessarily distinct) rays $\ell_i$ through $0\in \R^2$, then $\BC^\prime = \BC^\prime_0\times\R^{n-1}$, where $\BC^\prime_0 = \sum^{2Q+1}_{i=1}q^i_\#|\ell_i|$ for some rotations $q^i\in SO(2)$ with $|q^i-\id|<\epsilon$, for $\id:\R^2\to \R^2$ the identity map.

\textbf{Remark:} If $\BC\in \FL_I$, there exists $\epsilon = \epsilon(\BC)>0$ (in fact $\epsilon$ only depends on the smallest angle between the rays in the cross-section of $\BC$) such that every cone in $\FL_\epsilon(\BC)$ is of level at most $I$. \hfill $\blacktriangle$

The third condition in Definition \ref{defn:near} is necessary to ensure that varifolds in $\CN_{\epsilon}(\BC)$ are close to $\BC$ \textit{as varifolds}; the first two conditions only ensure closeness of $\spt\|V\|$ to $\spt\|\BC\|$, but not vice versa or necessarily with the same multiplicities. The following result qualitatively illustrates this ``closeness as varifolds'':

\begin{lemma}\label{lemma:convergence}
	Fix $\BC\in \FL$. Then if $V_i\in \CN_{\epsilon_i}(\BC)$ with $\epsilon_i\downarrow 0$, then $V_i\weakly \BC$ as varifolds in $B^{n+1}_1(0)$.
\end{lemma}

\begin{proof}
	It suffices to show that every subsequence of $(V_i)_i$ has a further subsequence which converges to $\BC$; let us therefore suppose we have already passed to a subsequence and we wish to show that another subsequence necessarily converges to $\BC$. By the compactness theorem for stationary integral varifolds, we know that we may pass to a subsequence (we will not relabel our indices, however) such that $V_i\weakly V$ for some stationary integral varifold $V\in \S_Q$. By definition of $\CN_{\epsilon_i}(\BC)$, we know that $\spt\|V\|\cap B_1\subset\spt\|\BC\|$; in particular, as $\BC$ comprises of half-hyperplanes and $\spt\|V\|$ has no boundary in $B_{1}$, we know that $\spt\|V\|$ comprises of half-hyperplanes also, which are a subcollection of the half-hyperplanes in $\BC$. Moreover, the last condition in the definition of $\CN_{\epsilon_i}(\BC)$ gives that necessarily each half-hyperplane of $\BC$ arises in $V$ with at least the same multiplicity. But our mass upper bound gives that these multiplicities must in fact coincide, i.e. $V = \BC$; this completes the proof.
\end{proof}

\subsection{Multiplicity $Q$ Classes}

In this section we will show that, for $\epsilon = \epsilon(n)>0$ sufficiently small, the class $\CN_{\epsilon}(\BC)$, for any $\BC\in \FL_S$, is contained within a so-called \textit{multiplicity $Q$ class}. Much of this section is similar to that seen in \cite{minter-5-2} for multiplicity two classes.

\begin{defn}
	We say a class $\M_Q$ is a \textit{multiplicity $Q$ class} if:
	\begin{enumerate}
		\item [(i)] Elements of $\M_Q$ are pairs $(V,U_V)$, where $U_V\subset\R^{n+1}$ is open and $V$ is a stationary integral $n$-varifold in $U_V$ which has stable regular part (in the sense of ($\S2$)) in $U_V$ and no classical singularities of density $<Q$;
		\item [(ii)] $\M_Q$ is closed under rotations and suitable homotheties, i.e. if $(V,U_V)\in \M_Q$, then for any orthogonal rotation $q$ of $\R^{n+1}$, point $X\in U_V$, and $\rho\in (0,\dist(X,\del U_V))$ we have $((q\circ\eta_{X,\rho})_\#V,(q\circ\eta_{X,\rho})(U_V))\in \M_Q$;
		\item [(iii)] Whenever $(V_j,U_j)\subset\M_Q$ and $U\subset\R^{n+1}$ is open such that $U\subset U_j$ for all sufficiently large $j$ and $\sup_{j\geq 1}\|V_j\|(K)<\infty$ for each compact $K\subset U$, then there is a subsequence $(V_{j'})_{j'}$ and $(V,U_V)\in \M_Q$ such that $U\subset U_{V}$, $V_{j'}\res U\weakly V\res U$, and moreover $\left.\Theta_V\right|_{U} \leq Q$ $\H^n$-a.e..
	\end{enumerate}
\end{defn}	

\textbf{Remark:} When $U_V$ is contextually clear, we shall simply write $V\in \M_Q$. \hfill $\blacktriangle$

The main estimate for multiplicity $Q$ classes is the following:

\begin{theorem}\label{thm:MQC-reg}
	Let $\Lambda>0$ and let $\M_Q$ be a multiplicity $Q$ class. Then, there exists a constant $\beta = \beta(\M_2,\Lambda)>0$ such that the following is true: if $(V,U_V)\in \M_Q$, $\rho>0$, $B_\rho(X_0)\subset U_V$, $\|V\|(B_\rho(X_0))\leq \Lambda$, $\spt\|V\|\cap B_{3\rho/4}(X_0)\neq\emptyset$, and $\rho^{-n-2}\int_{B_\rho(X_0)}\dist^2(X,P)\ \ext\|V\|(X)<\beta$ for some $n$-dimensional hyperplane $P\subset\R^{n+1}$, then either:
	\begin{enumerate}
		\item [(i)] There is a $q\in \{1,2,\dotsc,Q-1\}$ and $C^2$ functions $u_1,\dotsc,u_q:P\cap B_{\rho/2}(X_0)\to P^\perp$ such that $u_1\leq \cdots\leq u_q$, each $u_i$ solves the minimal surface equation, and $V\res B_{\rho/2}(X_0) = \sum_{i=1}^q |\graph(u_i)|\res B_{\rho/2}(X_0)$; or
		\item [(ii)] There is a $GC^{1,\alpha}$ function $u:P\cap B_{\rho/2}(X_0)\cap \A_Q(P^\perp)$ such that $V\res B_{\rho/2}(X_0) = \mathbf{v}(u)\res B_{\rho/2}(X_0)$;
	\end{enumerate}
	moreover, in either case we have
	$$\rho^{-2}\sup|u|^2 + \sup|Du|^2 \leq C\rho^{-n-2}\int_{B_\rho(X_0)}\dist^2(X,P)\ \ext\|V\|$$
	and in the case of \textnormal{(ii)}, the conclusions of Theorem \ref{thm:MW1} also hold; here $C = C(n,Q)\in (0,\infty)$ and $\alpha = \alpha(n,Q)\in (0,1)$.
\end{theorem}

\begin{proof}
	Follows by Theorem \ref{thm:MW1} and Theorem \ref{thm:MW2} using a contradiction argument identical to that in \cite[Theorem 2.17]{minter-5-2}.
\end{proof}

\begin{lemma}\label{lemma:M2C}
	Fix $\delta_0>0$. Then, there exists $\epsilon_0 = \epsilon_0(n,Q,\delta_0)$ such that if $\delta\geq \delta_0$ and $V$ is a stationary $n$-varifold in $B_1(0)$ which satisfies $\w_n^{-1}\|V\|(B_1(0))\leq Q + \frac{1}{2}+\delta$, then for any $X\in B_{\epsilon_0}(0)$ and any $\rho\in (0,1-|X|)$ we have
	$$\frac{\|V\|(B_\rho(X))}{\w_n\rho^n} \leq Q + \frac{1}{2} + 2\delta.$$
\end{lemma}

\begin{proof}
	From the monotonicity formula it follows that
	$$\frac{\|V\|(B_\rho(X))}{\w_n\rho^n} \leq \frac{\|V\|(B_{1-|X|}(X))}{\w_n(1-|X|)^n} \leq \frac{\|V\|(B_1(0))}{\w_n}\cdot\frac{1}{(1-|X|)^n} \leq \left(Q+\frac{1}{2}+\delta\right)\cdot \frac{1}{(1-\epsilon_0)^n}$$
	and since for any $a>0$ the map $y\mapsto \frac{a+2y}{a+y}$ is increasing for $y>0$, it suffices to take $\epsilon_0$ obeying $(1-\epsilon_0)^{-n}\leq \frac{Q+\frac{1}{2}+2\delta_0}{Q+\frac{1}{2}+\delta_0}$.
\end{proof}

One may now prove:

\begin{theorem}\label{thm:M2C}
	There exists $\epsilon_1 = \epsilon_1(n,Q)\in (0,1)$ and a multiplicity $Q$ class $\M_Q = \M_Q(n)$ such that the following is true: if $\BC\in \FL_S$ and $\epsilon\in (0,\epsilon_1]$, then $\CN_{\epsilon}(\BC)\subset\M_Q$, in the sense that there is a fixed open $U\supset B^{n+1}_{7/8}(0)$ for which $(V,U)\in \M_Q$ for each $V\in \CN_{\epsilon}(\BC)$.
\end{theorem}

\begin{proof}
	This follows by the same argument as in \cite[Theorem 2.19]{minter-5-2}, given Lemma \ref{lemma:M2C}, Theorem \ref{thm:MW1}, and Theorem \ref{thm:MW2}.
\end{proof}

\section{Density Gaps}\label{sec:density-gaps}

A key property we can prove about the class $V\in \CN_{\epsilon}(\BC)$ is that for $\epsilon = \epsilon(n,Q)\in (0,1)$ sufficiently small, for any $Z\in S(\BC)\cap B_{15/16}^{n+1}(0)$ points of density $\geq Q+1/2$ in $V$ will be close to $Z$, and indeed will accumulate at $Z$ as $\epsilon\downarrow 0$. More precisely, we define:

\begin{defn}
	Let $\delta>0$ and fix $\BC\in \FL$. We say that $V\in \S_Q$ has no $\delta$ \textit{density gaps} (with respect to $\BC$) if for each $y\in S(\BC)\cap B_{15/16}(0)\equiv \{0\}^2\times B^{n-1}_{15/16}(0)$ we have $B_\delta(y)\cap \{X:\Theta_V(X)\geq \Theta_\BC(0) = Q+1/2\} \neq\emptyset$.
\end{defn}

\begin{lemma}\label{lemma:gaps}
	Fix $\delta>0$ and $\BC\in \FL_S$. Then there exists $\epsilon = \epsilon(n,Q,\delta)\in (0,1)$ such that each $V\in \CN_\epsilon(\BC)$ has no $\delta$ density gaps (with respect to $\BC$).
\end{lemma}

\begin{proof}
	We will in fact prove more: we will show that there is an $\epsilon_* = \epsilon_*(n,Q)\in (0,1)$ sufficiently small such that for $\H^{n-1}$-a.e. $y\in B^{n-1}_{15/16}(0)$ we have
	$$\{X\in\spt\|V\|: \Theta_V(X)\geq Q+1/2\}\cap (\R^2\times\{y\}) \neq\emptyset.$$
	This of course proves the result, as for any $\delta>0$ one may apply Theorem \ref{thm:MW1} and Theorem \ref{thm:MW2} to find an $\epsilon = \epsilon(n,Q,\delta)\in (0,\epsilon_*)$ such that if $\BC\in \FL_S$, $V\in \CN_{\epsilon}(\BC)$, then on $B_{31/32}^{n+1}(0)\setminus \{|x|\leq \delta/4\}$ we have $\Theta_V\leq Q$, and thus any point of density $\geq Q+1/2$ must lie in $B_\delta(S(\BC))\cap B^{n+1}_{31/32}(0)$.
	
	We first claim that, for $\epsilon = \epsilon(n,Q)\in (0,1)$ sufficiently small, for every $Y = (0,y)\in \{0\}^2\times B_{31/32}^{n-1}(0)$, $\BC\in \FL_S$, and $V\in \CN_{\epsilon}(\BC)$, we have
	\begin{equation*}\tag{$\star$}\label{E:gap}
	\sing(V)\cap (\R^2\times\{y\})\cap B^{n+1}_{31/32}(0) \not\subset \CC_Q\cup \tilde{\B}_Q
	\end{equation*}
	where by $\CC_Q$ we mean the set of density $Q$ classical singularities in $V$, and by $\tilde{\B}_Q$ we mean the set of singular points in $V$ at which one tangent cone is a multiplicity $Q$ hyperplane (thus, by Theorem \ref{thm:MW1}, this must be the unique tangent cone at the point). We note that the proof of \eqref{E:gap} will be topological in nature, and so will not use the stationarity assumption of $V$ any more than through the structural properties of $\CC_Q$ and $\tilde{\B}_Q$.
	
	Indeed, to see this we argue as follows. Choose $\epsilon = \epsilon(n,Q)\in (0,1)$ such that if $V\in \CN_{\epsilon}(\BC)$, then on $\{(x,y)\in B^{n+1}_{63/64}(0):|x|>1/100\}$ we can apply Theorem \ref{thm:MW1} and Theorem \ref{thm:MW2} to express $V$ on this region as a sum of smooth minimal graphs or $Q$-valued $GC^{1,\alpha}$ graphs over subsets of the half-hyperplanes in some cone $\BC^\prime\in \FL_S$ which is close to $\BC$ (here, we are using the fact that cones in $\FL_S$ can only limit, as varifolds, onto cones in $\FL_S$ to ensure that $\epsilon$ can be chosen independent of $\BC$). Now, if \eqref{E:gap} failed for this choice of $\epsilon$, then we may find $\BC\in \FL_S$, $V\in \CN_{\epsilon}(\BC)$, and $(0,y_0)\in \{0\}^2\times B^{n-1}_{31/32}$ for which
	$$\sing(V)\cap (\R^2\times\{y_0\})\cap B^{n+1}_{31/32}(0)\subset \CC_Q\cup \tilde{\B}_Q.$$
	However, we know that $\CC_Q\cup \tilde{\B}_Q$ is a relatively open subset of $\sing(V)$ by Theorem \ref{thm:MW1} and Theorem \ref{thm:MW3}, and thus a simple compactness argument gives that we must be able to find a $\rho>0$ such that
	\begin{equation}\label{E:gap1}
	\sing(V)\cap (\R^2\times B^{n-1}_\rho(y_0))\cap B^{n+1}_{61/64}(0)\subset \CC_Q\cup \tilde{\B}_Q.
	\end{equation}
	Let $S:= (\R^2\times B^{n-1}_\rho(y_0))\cap B^{n+1}_{61/64}(0)$ denote this slab region. Note that one cannot apply Sard's theorem to find a slice of the form $\R^2\times\{y\}$, $y\in B^{n-1}_\rho(y_0)$, which intersects $V$ transversely (see below for precisely what we mean by this when $V$ is non-smooth) as $V$ is only at best $GC^{1,\alpha}$, which is not enough regularity for Sard's theorem to guarantee existence of a slice on which every point is a regular point of the projection map $\spt\|V\|\to B^{n-1}_\rho(y_0)$ (this would require suitable $C^2$ regularity of $V$)\footnote{Note that if one knows that $\tilde{\B}_Q$ has $\H^{n-1}$-measure 0 (as was the case in \cite{minter-5-2} when $Q=2$) then instead of \eqref{E:gap} one can instead assuming that the singular set is contained within $\CC_Q$ in a given slice, and then there will be enough regularity to use an argument based on Sard's theorem via projecting onto slices (as seen in \cite{minter-5-2} and \cite{simoncylindrical}). This is then enough to conclude the lemma due to the size assumption on the branch set.}. Also, note that there is no hope of proving orientability of $V$ in $S$, i.e. endowing $V\res S$ (or a suitable subset of) with a current structure, as it is a priori possible for $V$ to contain objects topologically similar to non-orientable smoothly immersed topological spaces such as Klein bottles.
	
	Instead, let $\gamma:S^1\to S$ be the smooth loop with image $\del B^2_{15/16}(0)\times \{y_0\}$; we know that $\gamma$ intersects $V$ act exactly $2Q+1$ points (by Theorem \ref{thm:MW1} and Theorem \ref{thm:MW2}), which by a slight perturbation of $\gamma$ we can assume are in $\reg(V)$ (as the regular set of $V$ is dense and open in $\spt\|V\|$ by Allard's regularity theorem), and moreover that these intersections are transversal to $V$; denote this perturbed curve by $\tilde{\gamma}$. For future reference, by being \textit{transverse} at a flat singular point we mean the usual transversality condition using the (uniquely defined) tangent hyperplane to $V$ at the point, and at a classical singularity we mean transverse to the boundary of each of the submanifolds-with-boundary in the classical singularity in the usual sense. 
	
	We know that this loop $\tilde{\gamma}$ is contractible in $S$ as $S$ is simply connected; let $F:[0,1]\times S^1\to S$ be a homotopy contracting $\tilde{\gamma}$ to a point in $S$ disjoint from $\spt\|V\|$, i.e. we have a smooth map $F: \overline{B}^2_1(0)\to S$ with $F|_{\del B^2_1(0)} = \tilde{\gamma}$. As locally about any point in $\spt\|V\|$, $V$ is locally a sum of a finite number of Lipschitz graphs, this Lipschitz regularity (along with the fact that $n\geq 2$) is enough regularity for us to apply Sard's theorem in this instance to prove a transversality homotopy theorem for $F$ (see \cite[Chapter 2]{GP} and \cite[Theorem 1.4]{BHS})\footnote{Indeed, the usual proof that one can homotopically change a non-transverse map to a nearby map which is transverse argues by characterising transversality of a certain family of nearby maps in terms of the existence of a regular value of another map, which in this case is a Lipschitz map from $\overline{B}^2_1(0)\to S^n$, where $S^n$ is the $n$-dimensional unit sphere. By Sard's theorem, namelt \cite[Theorem 1.4]{BHS}, under these assumptions we know that regular values exist, in the usual sense that every point in the pre-image of the regular value is a regular point, and so the argument still holds.}, i.e. there is a nearby (i.e. in supremum) smooth map $\tilde{F}:\overline{B}^2_1(0)\to S$ homotopic to $F$ which intersects $\spt\|V\|$ transversely. Note that as $F|_{\del B^2_1(0)} = \gamma$ intersects $\spt\|V\|$ transversely already, we can arrange our $\tilde{F}$ so that on an open neighbourhood $U \supset\del B^2_1(0)$ of $\del B^2_1(0)$ in $B^2_1(0)$ we have $\tilde{F}|_{U} = F|_U$.
	
	But now we define a new rectifiable set, $R$, by looking at the pre-image $\tilde{F}^{-1}(\spt\|V\|)\subset \overline{B}^2_1(0)$: we define the support of $R$ to be $\tilde{F}^{-1}(\spt\|V\|)$, and we take the multiplicity function on $R$ as: $\Theta_R(x):= \Theta_V(\tilde{F}(x))$; note that $\Theta_R$ is always integer-valued, and moreover lies in $\{1,2,\dotsc,Q\}$ at every point of $R$, and moreover from the regularity of $\spt\|V\|$ and transversality of $\tilde{F}$, $R$ has a unique tangent line at $\H^1$-a.e. point. The transversality condition of $\tilde{F}$, along with the local structural results of Theorem \ref{thm:MW1}, Theorem \ref{thm:MW2}, and Theorem \ref{thm:MW3}, guarantees that we may write $R = \reg(R)\cup \CC_R\cup \tilde{\B}_R$ as a disjoint union, where $\reg(R):= \tilde{F}^{-1}(\reg(V))$, $\CC_R :=\tilde{F}^{-1}(\CC_Q)\subset B^2_1(0)$, and $\tilde{\B}_R:= \tilde{F}^{-1}(\tilde{\B}_Q)\subset B^2_1(0)$; thus, locally about each point $x\in R$ by the transversality condition we have a radius $\rho_x>0$ such that:
	\begin{enumerate}
		\item [(i)] if $x\in \reg(R)\cap B_1^2(0)$, then $R\cap B_{\rho_x}(x)$ is a smooth 1-dimensional submanifold-with-boundary in $B_{\rho_x}(x)$, with boundary contained in $\del B_{\rho_x}(x)$, with some constant integer multiplicity; otherwise, if $x\in \reg(R)\cap \del B_1^2(0)$, then $R\cap B_{\rho_x}(x)$ is a smooth 1-dimensional submanifold-with-boundary (again of constant integer multiplicity) in $\overline{B}^2_1(0)\cap B_{\rho_x}(x)$, with one boundary point contained in $\del B^2_1(0)$ and the other in $\del B_{\rho_x}(x)\setminus \del B^2_1(0)$;
		\item [(ii)] if $x\in \CC_R$, then $R\cap B_{\rho_x}(x)$ is a smooth 1-dimensional classical singularity of density $Q$, i.e. it is a sum of $2Q$ smooth 1-dimensional submanifolds-with-boundary which either coincide or are disjoint except at a single point, which is necessarily the common interior boundary point of each of the $2Q$ curves;
		\item [(iii)] if $x\in \tilde{\B}_R$, then $R\cap B_{\rho_x}(x)$ is a $Q$-valued $GC^{1,\alpha}$ graph over some (affine) interval; in particular, it is a union of $Q$ Lipschitz graphs.
	\end{enumerate}
	In the above we are counting with multiplicity and so some of the sheets in a classical singularity may coincide, for example. We also know that $\CC_R$ will be a (not necessarily closed) discrete set of points in $B^2_1(0)$. We also know that there is some neighbourhood of $\del B^2_1(0)$ where every point of $R$ lies in $\reg(R)$, and that $\tilde{\B}_R$ is a compact subset of $B^2_1(0)$. Also, we have $\overline{\CC_R}\subset \CC_R\cup \tilde{\B}_R$.
	
	However, we can now derive the necessary contradiction to prove \eqref{E:gap}, as we are in a situation where $R$ has, counted with multiplicity, $2Q+1$ boundary points in $\del B^2_1(0)$, yet in the interior $B^2_1(0)$ is a locally about every point a union of finitely many Lipschitz graphs which has no interior boundary points. Indeed, we can do this by ``unravelling'' curves in $R$, essentially arguing as one would in the classification of compact smooth $1$-manifolds; here we must be more careful with multiplicity, but the idea is essentially the same.
	
	Let $P = \{p_1,\dotsc,p_{2Q+1}\}$ denote the $2Q+1$ points on $\del B^2_1(0)$ (where we count each point a distinct number of times depending on its multiplicity in $R$). Firstly, cover $\tilde{\B}_R$ by finitely many balls, $B_1,\dotsc,B_N$, such that $B_i\subset B^2_1(0)$ and for each $i$, $R\cap B_i$ is a union of $Q$ Lipschitz graphs. Note that, as we know $\left|\CC_R\cap B^2_1(0)\setminus \cup_i B_i\right|<\infty$ and $\del(\cup_i B_i)\subset \reg(R)\cup \CC_R$, we may decrease the radius of each ball $B_i$ slightly to ensure that $\CC_R\cap \del B_i = \emptyset$ for each $i$, and so in particular $\del (\cup_i B_i)\subset \reg(R)$, and that the $B_i$ will still cover $\tilde{\B}_R$.
	
	We now define to the notion of an $R$\textit{-path}; loosely speaking, this is a continuous curve $\Gamma:[0,1]\to R$ which obeys $\Gamma(0)\in P$ which traverses $R$ without back-tracking, and is only allowed to pass through a given point $x\in R$ at most $\Theta_R(x)$-times. More precisely, a continuous path $\Gamma:[0,1]\to R$ is an $R$\textit{-path} if it obeys all the following:
	\begin{itemize}
		\item $\Gamma(0)\in P$;
		\item if $\Gamma(t)\in \reg(R)$, then there is a $\delta>0$ such that $\Gamma|_{(t-\delta,t+\delta)\cap [0,1]}$ is injective;
		\item for each $x\in R$, $|\{t:\Gamma(t) = x\}|\leq \Theta_R(x)$;
		\item if $\Gamma(t)\in \cup_i B_i$, then, if $I$ is the largest interval containing $t$ which obeys $\Gamma|_{I}\subset \cup_i B_i$, then for each $t_0\in I$, there is a $\delta = \delta(t_0)>0$ such that $\Gamma|_{(t_0-\delta,t+\delta_0)\cap [0,1]}$ is injective.
	\end{itemize}
	Note that in the last bullet point, the existence of such a largest interval $I$ is guaranteed by continuity of $\Gamma$, and moreover it will always be of the form $(a,b)$ or $(a,1]$. We also note that the local injectivity conditions ensures that $\Gamma$ ``keeps moving forward'' in such regions. In particular, the only points where $\Gamma$ need not be locally injective is about points in $\CC_R\setminus\cup B_i$, which is a finite set.
	
	Fix a boundary point $p\in P$. We now claim that one can find an $R$-path $\Gamma$ starting at $p$ such that $\Gamma(1)\in P\setminus\{p\}$. Indeed, as $p\in \reg(R)$, we know that one can define an $R$-path $\Gamma_1$ with $\Gamma_1(0) = p$ simply by allowing $\Gamma_1$ to traverse the connected component $A_1$ of $\reg(R)\setminus \cup_i B_i$ containing $p$. At this point, now define a new set, $R_1$, which is identical to $R$ except we decrease the density of $A_1$ by $1$, i.e. $\Theta_{R_1}(x) := \Theta_{R}(x)-1$ for all $x\in A_1$, and if $x\not\in A_1$ then $\Theta_{R_1}(x):= \Theta_{R}(x)$; if this causes some $x\in R_1$ to have density $0$, we remove it from $R_1$. We then know that $\Gamma_1(1) \in (\reg(R)\cap \del B^2_1(0))\cup (\CC_R\setminus \cup_i B_i)\cup \del( \cup_i B_i)$: we extend $\Gamma_1$ to a new path $\Gamma_2$, and define a new set $R_2$ from $R_1$, in each of these cases by:
	\begin{enumerate}
		\item [(a)] if $\Gamma_1(1) \in \reg(R)\cap \del B^2_1(0)$, then set $\Gamma_2:= \Gamma_1$ and take $R_2$ to be identical to $R_1$;
		\item [(b)] if $\Gamma_1(1)\in \CC_R\setminus \cup_i B_i$, then locally about $\Gamma_1(1)$ we know that $R_1$ is a sum of $2Q-1$ submanifolds-with-boundary (as we decreased the density of one such submanifold-with-boundary by $1$). Simply choose any of these $2Q-1$ submanifolds-with-boundary, and extend $\Gamma_1$ to $\Gamma_2$ by traversing the connected component $A_2$ of $\reg(R)\setminus\cup_i B_i$ containing that submanifold-with-boundary. Then, define $R_2$ from $R_1$ by decreasing the multiplicity of $A_2$ by 1;
		\item [(c)] if $\Gamma_1(1)\in \del (\cup_i B_i)$, then choose a Lipschitz graph in $\cup_i B_i$ which starts at $\Gamma_1(1)$, and extend $\Gamma_1$ via this path. Keep on repeating this (and only finitely many repetitions are needed as there are only finitely many balls and each ball only have a finite number of Lipschitz graphs) until the path extension leaves $\cup_i B_i$; this point of exiting will necessarily be a regular point of $R$ by construction, and so we may further extend the path by following this connected component of the regular set in $\reg(R)\setminus\cup_i B_i$. We then define $\Gamma_2$ to be this new path, and define the new set $R_2$ by decreasing the density of each point in $\overline{\cup_i B_i}$ by the number of distinct times that $\Gamma_2$ traversed it.
	\end{enumerate}
	Note that, every point in $R_2$ still has integer density after this procedure; in particular, every classical singularity in $R_2$ is still comprised of an even number of submanifolds-with-boundary (counted with multiplicity), and in each $B_i$ we have that $R_2$ is a union of at most $Q$ Lipschitz graphs (which may be different to the original Lipschitz graphs defined in $B_i$ if some of the balls $B_i$ intersect).
	
	Moreover, we know that $\Gamma_2(1)\in (\reg(R)\cap \del B^2_1(0))\cup (\CC_R\setminus \cup_i B_i)\cup \del(\cup_i B_i)$, and thus we are in the same situation as before for $\Gamma_1$, except we have a new set $R_2$; however, $R_2$ has all the desired properties needed for inductively continuing this procedure, namely that: if $\Gamma_2(1)\in \reg(R)\cap \del B^2(0)$, we terminate the inductive procedure; if $\Gamma_2(1)\in \CC_R\setminus\cup_i B_i$, then necessarily $\Gamma_2(1)$ is a classical singularity with an odd (and hence non-zero) number of submanifolds-with-boundary in $R_2$, and so in particular there is another sheet of the classical singularity we can choose to continue the procedure; and if $\Gamma_2(1)\in \del (\cup_i B_i)$, then there must be at least one continuous path in $\cup_i B_i$ which we can use to extend $\Gamma_2$, as if not then, as all points in $R\cap \cup_i B_i$ are regular points, this would mean that the regular piece we used to connect $\Gamma_2$ to this endpoint occurs with multiplicity $0$ in $R_1$, which is a contradiction by our construction.
	
	Hence, we may iteratively continue this procedure, constructing a sequence $(\Gamma_i,R_i)$, where each $\Gamma_{i+1}$ is an extension of the path $\Gamma_i$, and $R_{i+1}$ is constructed from $R_i$ by subtracting from the density of a given point $x\in R_i$ the number of times $\Gamma_{i+1}$ traverses $x$. In particular, as there are only finitely many balls $B_1,\dotsc,B_N$, all of which start with a finite number (i.e. $Q$) Lipschitz graphs in $R$, and finitely many classical singularities away from $\cup_i B_i$ in $R$, say $\{c_1,\dotsc,c_M\}$, and thus finitely many connected components of $\reg(R)\setminus\cup_i B_i$, say $\{A_1,\dotsc,A_K\}$, we see that, if we write $q_i \equiv (q^\ell_i)_{\ell=1}^L:= (q_i^{B_1},\dotsc,q_i^{B_N}, q_i^{c_1},\dotsc,q_i^{c_M},q^{A_1}_i,\dotsc,q^{A_K}_i)$ for the vector of (integer) densities of these respective regions and points in $R_i$, then the sum $\sum^L_{\ell=1}q_i^\ell$ is strictly decreasing in $i$, and this sum is finite (as it has an upper bound of $(N+M+K)Q$) and always non-negative, and hence this procedure must terminate in finitely many steps\footnote{As only finitely many steps are needed and at each stage there are only finitely many choices to make, we do not need to invoke the axiom of choice in order to make these choices. Of course, if one was happy to argue via the axiom of choice, namely via Zorn's lemma, the existence of a maximal path can be deduced by introducing a suitable partial order and arguing that every chain has an upper bound. Maximality of the path then forces it to end in $R\cap \del B^2_1(0)$.}; in particular, to terminate the path must end at a point in $\del B^2_1(0)$, and so there is a $i_*\geq 1$ such that $\Gamma_{i_*}(1) \in \reg(R)\cap \del B^2_1(0)$, and moreover by construction we always have that $\Gamma_{i_*}(1)\in P\setminus \{\Gamma_{i_*}(0)\}$ (as, if we had $\Gamma_{i_*}(1)$ ending at the same point as its starting point, then the component of $\reg(R)$ containing this point must have multiplicity $\geq 2$, and so we can choose $\Gamma_{i_*}(1)$ to be distinct from $p = \Gamma_{i_*}(0)$). The path $\Gamma_{i_*}$ is the desired path $\Gamma$.
	
	But now look at $R_{i_*}$; this has the same form as $R$, namely it has an odd number (i.e. $2Q-1$) points on $\del B^2_1(0)$ (which moreover are regular points), say $P_1:= \{p_1^1,\dotsc,p_{2Q-1}^1\}$ (which will be $P\setminus \{p,\Gamma_{i_*}(1)\}$) and in the interior it has classical singularities (which by construction have an even number of sheets, counted with multiplicity) and flat singular points, where locally $R_{i_*}$ is represented by a sum of (at most $Q$) Lipschitz graphs over an interval. Thus, the argument may be repeated for $R_{i_*}$, starting at some point $p_*\in P_1$. If we inductively repeat this procedure $Q$ times, we end up with some set $R^\prime$ of the same general form, except $R^\prime$ has exactly one (regular) point in $\del B^2_1(0)$ (which moreover has multiplicity one). However this is now a contradiction, as if we perform the same construction of an $R^\prime$-path starting at this boundary point, we see that in finitely many steps the procedure must terminate, and necessarily from the structure of $R^\prime$ the only way this path can terminate is if it ends at a point in $R^\prime\cap \del B^2_1(0)$; but this is impossible as there was only one boundary point, and thus we reach a contradiction. This contradiction therefore proves \eqref{E:gap}. 
	
	From \eqref{E:gap} we can now prove the result. Indeed, with $\epsilon = \epsilon(n,Q)\in (0,1)$ as in \eqref{E:gap}, note that \eqref{E:gap} gives that for each $y\in B^{n-1}_{31/32}(0)$, we must be able to find some $X_y\in \sing(V)\cap (\R^2\times\{y\})\cap B^{n+1}_{31/32}(0)$ with $X_y\not\in \CC_Q\cup \tilde{\B}_Q$. But, from the usual dimension bounds on the stratification of the singular set of $V$ we know that $\dim_\H(\S_{n-2})\leq n-2$, where $\S_{n-2}$ is the $(n-2)$-strata, and hence this gives that for $\H^{n-1}$-a.e. $y\in B^{n-1}_{31/32}(0)$, the corresponding $X_y$ must lie in $\sing(V)\setminus (\S_{n-2}\cup \CC_Q\cup \tilde{\B}_Q)$. At such $X_y$, there must be a tangent cone to $V$ at $X_y$ which has spine of dimension at least $n-1$. However, from Theorem \ref{thm:MW1}, Theorem \ref{thm:MW2}, and Theorem \ref{thm:MW3} and the fact $X_y\not\in \CC_Q\cup \tilde{\B}_Q$, this implies that this tangent cone must be at sum of at least $2Q+1$ half-hyperplanes with a common boundary, which implies that $\Theta_V(X_y)\geq Q+1/2$; this completes the proof. 
\end{proof}


\section{Boundary Regularity of Blow-Up Classes}\label{sec:blow-up}

The aim of this section is to show that, given it obeys appropriate properties (see $(\FB1)-(\FB8)$ below), each function within a blow-up class relative to a stationary classical cone $\BC^{(0)}$ is $GC^{1,\gamma}$ up-to-the-boundary for some uniform $\gamma = \gamma(n,Q)\in (0,1)$. The main difference in this section compared to that in \cite{minter-5-2} is that we no longer have such a nice structure for the blow-ups, and so one must work harder to establish their boundary regularity. The resolution is to apply a combination of the techniques seen in \cite{minter-5-2} and \cite{minterwick}.

Let us fix a base cone $\BC^{(0)}\in \FL_S\cap \FL_I$, where $I\in \{0,1,2\}$; thus, we may write $\BC^{(0)} = \BC^{(0)}_0\times\R^{n-1}$, and $\BC^{(0)} = \sum_{i=1}^N q_i\llbracket H_i^{(0)}\rrbracket$, where $H_i^{(0)}$ are half-hyperplanes in $\R^{n+1}$ with $\del H_i^{(0)} = \{0\}^2\times\R^{n-1}$ and $q_i\in \{1,2,\dotsc,Q\}$ are such that $\sum_{i=1}^Nq_i = 2Q+1$. Let us assume without loss of generality that our labelling is chosen so that $q_1\leq q_2\leq \cdots\leq q_N$ (so in particular, when $I>0$, $q_{N-I+1}=\cdots = q_{N} = Q$, and otherwise all $q_i$ are $<Q$).

For notational convenience, we rotate each half-hyperplane $H^{(0)}_i$ to the fixed half-hyperplane $H:= \{(x^1,\dotsc,x^{n+1})\in \R^{n+1}:x^1=0,\; x^2>0\}$ whilst fixing the axis $\{0\}^2\times\R^{n-1}$; this is done via finding a rotation $\tilde{R}_i:\R^2\to \R^2$ which maps the ray $\ell_i$ in $\spt\|\BC^{(0)}_0\|$ corresponding to $H^{(0)}_i$ to the ray $\{(x^1,x^2)\in \R^2:x^1=0,\; x^2>0\}$ and then extending $\tilde{R}_i$ to a rotation $R_i:\R^{n+1}\to \R^{n+1}$ via $R_i(x,y):= (\tilde{R}_i(x),y)$. Moreover, any function $\tilde{v}:H_i\to H_i^\perp$ then gives a function $v:H\to H^\perp$ via $v(x):= R_i v(R_i^{-1}x)$. We do this so that we can treat a function defined on $\spt\|\BC^{(0)}\|$ instead as tuple of functions defined on $H$.

\begin{defn}\label{def:blow-up}
	We say a collection of functions $\FB(\BC^{(0)})$ is a \textit{proper} (\textit{coarse}) \textit{blow-up class} over $\BC^{(0)}\in \FL_S\cap \FL_I$ as above if it obeys the following properties:
	\begin{enumerate}
		\item [$(\FB1)$] Each element $v\in \FB(\BC^{(0)})$ takes the form $v = (v^1,\dotsc,v^{N})$, where for $i=1,\dotsc,N-I$ $v^i\in L^2(B_1\cap H; \R^{q_i})\cap W^{1,2}_{\text{loc}}(B_1\cap H;\R^{q_i})$ and for $i=N-I+1,\dotsc,N$, $v^i\in L^2(B_1\cap H;\A_{Q}(H^\perp))\cap W^{1,2}_{\text{loc}}(B_1\cap H;\A_Q(H^\perp))$;
		\item [$(\FB2)$] If $v\in \FB(\BC^{(0)})$, we have that if $i\leq N-I$ then $v^i = (v^i_1,\dotsc,v^i_{q_i})$ is harmonic with $v^i_1\leq \cdots\leq v^i_{q_i}$, and if $i>N-I$, then $v^i$ obeys in $H\cap B_1$ the conclusions of Theorem \ref{thm:MW4}; in particular it is $GC^{1,\alpha}$ in $H\cap B_1$, for some $\alpha = \alpha(n,Q)$;
		\item [$(\FB3)$] If $v\in \FB(\BC^{(0)})$ and $z\in B_1\cap \del H$, then for each $\rho\in (0,\frac{3}{8}(1-|z|)]$ we have:
		$$\int_{B_{\rho/2}(z)\cap H}\sum^{N}_{i=1}\frac{|v^i(z)-\kappa^i(z)|^2}{|x-z|^{n+3/2}}\ \ext x \leq C\rho^{-n-3/2}\int_{B_\rho(z)\cap H}\sum^N_{i=1}|v^i(x)-\kappa^i(z)|^2\ \ext x$$
		where $\kappa:B_1(0)\cap \del H\to \R^2$ is a smooth single-valued function and $\kappa^i(x) := R_i \kappa^{\perp_{H_i}}(Q^{-1}_i(x))$, thus $|v^i-\kappa^i(z)|^2 = \sum_{j=1}^{q_i}|v^i_j-\kappa^i(z)|^2$; moreover $\kappa$ obeys
		$$\sup_{B_{5/16}\cap \del H}|\kappa|^2\leq C\int_{B_{1/2}\cap H}|v|^2;$$
		\item [$(\FB4)$] If $v\in \FB(\BC^{(0)})$, $z\in B_1\cap \del H$, and $\rho\in (0,\frac{3}{8}(1-|z|)]$, we have:
		$$\int_{B_{\rho/2}(z)\cap H}\sum^N_{i=1}R_z^{2-n}\left(\frac{\del}{\del R_z}\left(\frac{v^i-v^i_a(z)}{R_z}\right)\right)^2\ \leq C\rho^{-n-2}\int_{B_{\rho/2}(z)\cap H}\sum^N_{i=1}|v^i-\ell_{v^i,z}|^2$$
		where $R_z(x):= |x-z|$ and $\ell_{v^i,z}(x):= v^i_a(x) + (x-z)\cdot Dv^i_a(z)$ is the first-order linear approximation to the average part $v^i_a:= q_i^{-1}\sum_{j=1}^{q_i}v^i_j$ at $z$;
		\item [$(\FB5)$] If $v\in \FB(\BC^{(0)})$, then:
		\begin{enumerate}
			\item [$(\FB5\text{I})$] For each $z\in B_1\cap \del H$ and $\sigma\in (0,\frac{3}{8}(1-|z|)]$, if $v\not\equiv 0$ in $B_\sigma(z)\cap H$, then $v_{z,\sigma}(\cdot):= \|v(z+\sigma(\cdot))\|^{-1}_{L^2(B_1\cap H)}|v(z+\sigma(\cdot))\in \FB(\BC^{(0)})$;
			\item [$(\FB5\text{II})$] $\|v-\ell_v\|^{-1}_{L^2(B_1\cap H)}(v-\ell_v)\in \FB(\BC^{(0)})$ whenever $v-\ell_v\not\equiv 0$ in $B_1\cap H$, where $v-\ell_v:= (v^1-\ell_{v^1},\dotsc,v^N-\ell_{v^N})$ and $\ell_{v^i}\equiv \ell_{v^i,0}$ (from $(\FB4)$);
		\end{enumerate}
		\item [$(\FB6)$] If $(v_m)_m\subset\FB(\BC^{(0)})$, then there is a subsequence $(m')\subset(m)$ and a function $v\in \FB(\BC^{(0)})$ such that $v_{m'}\to v$ strongly in $L^2_{\text{loc}}(B_1\cap\overline{H})$ and weakly in $W^{1,2}_{\text{loc}}(B_1(0)\cap H)$;
		\item [$(\FB7)$] There exist constants $\beta = \beta(n,Q)$ and $\epsilon = \epsilon(\BC^{(0)})$ such that whenever $v\in \FB(\BC^{(0)})$ has $v^i_a(0) = 0$, $Dv^i_a(0) = 0$ for each $i=1,\dotsc,N$, and $\|v\|_{L^2(B_1\cap H)}=1$, then the following is true: if $v_* = (v_*^1,\dotsc,v^{N}_*)$ is such that for $i\leq N-I$, $v^i_*\equiv 0$ and for $i>N-I$, $v^i_*:H\to \A_Q(H^\perp)$ is of the form $v^i_*(x) = \sum_{j=1}^Q \llbracket a^i_j x^2\rrbracket$ with $(v^i_*)_a\equiv 0$, where $a_j^i\in \R$, yet for some $i>N-I$ we have $v^i_*\not\equiv 0$, and moreover if
		$$\int_{B_1\cap H}\G(v,v_*)^2 < \epsilon$$
		then we have $\left. v\right|_{B_{1/2}\cap H}\in GC^{1,\beta}(\overline{B_{1/2}(0)\cap H})$;
		\item [$(\FB8)$] Fix $\Omega\subset H\cap B_1$ an open subset. Suppose $v\in \FB(\BC^{(0)})$ is $GC^1$ on $\overline{\Omega}$, and obeys $v_a\equiv 0$. Then, for any $\zeta\in C^1(\overline{\Omega}\cap B_{3/4};\R)$ with $\spt(\zeta)$ a compact subset of $\overline{\Omega}\cap B_{3/4}^n(0)$, we have
		$$\int_{H\cap B_1}|Dv_i|^2\zeta = -\int_{H\cap B_1}\sum_{\alpha=1}^{q_i}v_i^\alpha Dv_i^\alpha \cdot D\zeta$$
		and for any $\psi = (\psi^2,\dotsc,\psi^{n+1})\in C^1(\overline{\Omega}\cap B_{3/4};\R^n)$ with $\spt(\psi)$ a compact subset of $\overline{\Omega}\cap B_{3/4}^n(0)$ and furthermore with $\left.\psi^2\right|_{\{x^2=0\}}\equiv 0$:
		$$\int_{H\cap B_1}\sum_{\alpha=1}^{q_i}\sum_{j,k=1}^n\left(|Dv_i|^2\delta_{jk} - 2D_jv_i^\alpha D_kv_i^\alpha\right)D_j\psi^k = 0;$$
		here, $i=1,\dotsc,N$.
	\end{enumerate}
\end{defn}

We remark that the new property $(\FB8)$ actually follows immediately from $(\FB2)$, $(\FB3)$, and $(\FB4)$ (see Section \ref{sec:squash-squeeze}), and is therefore not a new variational identity arising from a suitable application of the first variation formula for the sequence of stationary integral varifolds generating $v\in \FB(\BC^{(0)})$; as such, one does not need new arguments for the coarse and fine blow-up situations.

The main result of this section is the following $GC^{1,\gamma}$ boundary regularity statement for $v\in \FB(\BC^{(0)})$:

\begin{theorem}\label{thm:boundary-reg}
	Let $\BC^{(0)}\in \FL_S\cap \FL_I$, where $I\in \{0,1,2\}$. Then, there exists $\gamma = \gamma(n,Q)\in (0,1)$ such that if $v\in \FB(\BC^{(0)})$, then $v\in GC^{1,\gamma}(\overline{H\cap B_{1/8}(0)})$. Moreover, $\left.v_s\right|_{\del H}\equiv 0$, and we have the estimate:
	$$\rho^{-n-2}\int_{B_{\rho}(z)\cap H}\G(v,\ell_z)^2\leq C\rho^{2\gamma}\int_{B_{1/2}(0)\cap H}|v|^2$$
	for every $\rho\in (0,1/8]$ and $z\in \overline{H}\cap B_{1/8}$; here, $\ell_z = (\ell^1_z,\dotsc,\ell_z^N)$ is $\ell_z^i:= v^i_a(z) + (x-z)\cdot Dv^i(z)$, and $C = C(n,Q)\in (0,\infty)$.
\end{theorem}

In the same way as seen in \cite{minterwick} and \cite{minter-campanato} (which we do not detail and refer the reader to \cite{minterwick}), this result will follow once one proves a suitable classification of the homogeneous degree one elements of $\FB(\BC^{(0)})$, namely:

\begin{lemma}\label{lemma:classification}
	Suppose that $v\in \FB(\BC^{(0)})$ is homogeneous of degree one in $H\cap B_1$. Then, $v$ is a linear function on $\overline{H}\cap B_1$, and moreover the average-free part $(v_i)_f$ takes the form $x\mapsto \sum^{q_i}_{\alpha=1}\llbracket a^\alpha_i x^2\rrbracket$ for each $i$, where $a^\alpha_i\in \R$.
\end{lemma}

\begin{proof}
	Following the corresponding argument in \cite{minter-campanato}, one may deduce the following. Suppose $v\in \FB(\BC^{(0)})$ is homogeneous of degree 1. Then by $(\FB3)$ we have $v\in C^{0,\mu}(\overline{H \cap B_1})$ for some $\mu = \mu(n,Q)$, and for each $i\in\{1,2,\dotsc,N\}$, the average $(v_i)_a$ is necessarily a linear function as it is harmonic; thus by $(\FB5\text{II})$ we may assume that $(v_i)_a\equiv 0$ for all $i$, and so in particular from $(\FB4)$ it now follows that $\left.v_i\right|_{B_1\cap \del H}\equiv q_i\llbracket 0\rrbracket$; in particular when $q_i<Q$, as $v_i^\alpha$ is harmonic and smooth up to the boundary with zero boundary values, we may reflect it across $\del H$ to see that it is necessarily linear. Without loss of generality extend $v$ to all of $H$ by homogeneous degree one extension. Now if we set
	$$T_i(v):=\{z\in \overline{H}:v^i(z+x) = v^i(x)\text{ for all }x\in\overline{H}\}$$
	then we must have $T_i(v)\subset \del H$, and $\dim(T_i(v))\leq n-1$; if $\dim(T_i(v)) = n-1$ then $v$ takes the desired form, so suppose $\dim(T_i(v))\leq n-2$ for some $i$; this is what we need to contradict to prove the result. If we write $d_i(v):= \dim(T_i(v))$, choose $v$ such that $\sum_i d_i(v)$ is maximal with the property that $d_i(v)\leq n-2$ for some $i\in \{1,\dotsc,N\}$. Then, one can show that for any $i_*$ with $d_{i_*}(v)\leq n-2$, that $v_{i_*}$ is $GC^1(\overline{U})$ for each bounded open set $U\subset H\setminus T_{i_*}(v)$ with $\dist(\overline{U},T_{i_*}(v))>0$, i.e. $v_{i_*}$ is $GC^1$ up-to-the-boundary of $H$ except on its spine $T_{i_*}(v)$. (This finishes the recall from \cite{minter-campanato}.)
	
	Thus, one may deduce from $(\FB8)$ that for any $\Omega\subset H\cap B_1$ open with $\dist(\overline{\Omega},T_{i_*}(v))>0$ that for any $\zeta\in C^1_c(\overline{\Omega}\cap B_{3/4};\R)$:
	$$\int_{H\cap B_1}|Dv_i|^2\zeta = -\int_{H\cap B_1}\sum^{q_i}_{\alpha=1}v_i^\alpha Dv_i^\alpha\cdot D\zeta$$
	and for any $\psi = (\psi^2,\dotsc,\psi^{n+1})\in C^1_c(\overline{\Omega}\cap B_{3/4};\R^n)$ with $\left.\psi^2\right|_{\{x^2=0\}}\equiv 0$:
	$$\int_{H\cap B_1}\sum^{q_i}_{\alpha=1}\sum^n_{j,k=1}\left(|Dv_i|^2\delta_{jk}-2D_jv_i^\alpha D_kv^\alpha_i\right)D_j\psi^k = 0$$
	for each $i$ such that $d_i(v)\leq n-2$ (we need not worry when $d_i(v)=n-1$ as $v_i$ then already takes the desired form). But for such $i$, $T_i(v)$ has vanishing $2$-capacity and so by a standard excision argument we can extend this to any $\zeta\in C^1_c(\overline{B_{3/4}(0)};\R)$ and $\psi = (\psi^2,\dotsc,\psi^{n+1})\in C^1_c(\overline{B_{3/4}(0)};\R^n)$ with $\left.\psi^2\right|_{\{x^2=0\}}\equiv 0$. Armed with these, we can now use standard argument (see e.g. \cite{simonwick-frequency}) that the frequency $N_{v_i}(y)$ is well-defined at every point $y\in \overline{H}$; it should be noted that in these arguments we take $\psi$ to be of the form $\psi^j = x^j\phi(|x|)$ for a suitable cut-off function $\phi$, and thus the condition $\left.\psi^2\right|_{\{x^2=0\}}$ is obeyed here\footnote{Note that we may perform an odd reflection of $v$ across $\del H$, in which case the frequency function which appears to be at a boundary point is in fact at simply the frequency at an interior point of the reflected function.}. In particular, as $v_i$ is homogeneous of degree one we know $N_{v_i}(0) = 1$, and moreover by standard properties of frequency monotonicity we know that if $y\in \overline{H}$ has $N_{v_i}(y)\geq 1$, then $v_i$ is translation invariant with respect to directions parallel to $y$, i.e. $y\in T_{i}(v)$. 
	
	But now $(\FB4)$ tells us that every branch point $y\in \B_{v_i}\cap (\del H\setminus T_{i}(v))$ has $N_{v_i}(y)\geq 1$, and thus must belong to $T_{i}(v)$; moreover we know the same is true for interior branch points by Theorem \ref{thm:MW4}. Moreover by $(\FB7)$ and Theorem \ref{thm:MW4}, the same is true for any classical singularities $y\in \overline{H}\setminus T_{i}(y)$, and thus we see that $v_i$ is regular (and thus classically harmonic) on $\overline{H}\setminus T_{i}(v)$. Hence, after we quotient out by the translation invariant subspace $T_{i}(v)$, each $v_i$ with $d_i(v)\leq n-2$ is determined by a homogeneous of degree one function $g_i:H_i\to \A_{Q}(\R)$, where $H_i$ is the half-space $\{(x^1,\dotsc,x^k):x^1>0\}$ in $\R^{k}$ for some $k\geq 2$, and moreover $g_i$ is smoothly harmonic on $\overline{H}_i\setminus 0$ with $|g_i|\equiv 0$ on $\del H_i$, and $C^{0,\mu}$ on $\overline{H}_i$. Thus, if we apply an odd reflection across $\del H$, we see that $g_i$ can be extended a function on $\R^k$ which is harmonic on $\R^k\setminus\{0\}$ and continuous on all of $\R^k$. As $k\geq 2$ and $g_i$ is continuous at $0$, the point $0$ is a removable singularity for $g_i$, and thus $g_i$ is extendable to a homogeneous degree one harmonic function on all of $\R^k$. But then $g_i$ must be linear, and as $|g_i|\equiv 0$ on $\del H_i$, it follows that $g_i$, and hence $v_i$, takes the desired form; in particular $d_i(v) = n-1$, a contradiction to our assumption. This completes the proof.
\end{proof}

\section{Boundary Squash and Squeeze Identities for Blow-Ups}\label{sec:squash-squeeze}

Given the results of the previous sections, the proof of Theorem \ref{thm:A} follows the same general argument as seen in \cite{minter-5-2}. Of course, the one difference is that we must establish the new additional property of the coarse blow-up class $\FB(\BC^{(0)})$ and the fine blow-up class $\FB_{3,1;M}(\BC^{(0)})$, namely $(\FB8)$; this will be the focus of the present section. To avoid repetition, we refer the reader to \cite{minter-5-2} for precise definitions of the coarse and fine blow-up classes, as they are simple adaptations of those seen there.

\textbf{Notation:} For $K\subset \overline{H}\cap B_1$ a compact subset, we write $C^1_K(B_1)$ for the $C^1(B_1)$ functions which have support contained in $K$. We write $C^1_{c}(\overline{H}\cap B_1)$ for the functions which belong to $C^1_K(B_1)$ for some compact $K\subset \overline{H}\cap B_1$.

We will prove the following results:

\begin{lemma}[$GC^1$ Boundary Squash Identity]
	Let $v\in \FB(\BC^{(0)})$ and suppose $v$ is $GC^1$ on some $\overline{\Omega}\subset\overline{H}\cap B_1$, where $\Omega\subset H\cap B_1$ is open. Then, for any $\zeta\in C^1_c(\overline{\Omega}\cap B_1)$ we have for each $i=1,\dotsc,N$,
	$$\int_{H\cap B_1}|D(v_i)_f|^2\zeta = -\int_{H\cap B_1}\sum_\alpha (v_i)_f^\alpha D(v_i)_f^\alpha\cdot D\zeta.$$
\end{lemma}

\begin{proof}
	Let $\eta_\epsilon:\R\to \R$ be a smooth even function obeying $\eta_\epsilon\equiv 0$ on $[0,\epsilon)$, $\eta_\epsilon\equiv 1$ on $[2\epsilon,\infty)$, and $|\eta^\prime_\epsilon|\leq C\epsilon^{-1}$, where $C \in (0,\infty)$ is an absolute constant. Let $\zeta\in C^1_c(\overline{\Omega}\cap B_1)$. Then we know that $\zeta(x)\eta(x^2)$ is a suitable test function for the squash identity for any given $v_i$, i.e.
	$$\int_{\Omega\cap B_1}|D(v_i)_f|^2\zeta\eta_\epsilon = -\int_{\Omega\cap B_1}\sum_\alpha \eta_\epsilon(x^2)(v_i)_f^\alpha D(v_i)_f^\alpha\cdot D\zeta + \zeta\eta^\prime_\epsilon(x^2)(v_i)_f^\alpha D_2(v_i)^\alpha_f.$$
	As $v_i$ is $GC^1$ on $\overline{\Omega}$, when we take $\epsilon\downarrow 0$ we have by the dominated convergence theorem,
	$$\int_{\Omega\cap B_1}|D(v_i)_f|^2\zeta\eta_\epsilon\to \int_{\Omega\cap B_1}|D(v_i)_f|^2\zeta$$
	and
	$$\int_{\Omega\cap B_1}\sum_\alpha \eta_\epsilon(x^2)(v_i)_f^\alpha D(v_i)_f^\alpha \cdot D\zeta\to \int_{\Omega\cap B_1}\sum_\alpha (v_i)^\alpha_f D(v_i)_f^\alpha\cdot D\zeta$$
	which are the desired terms; so we just need to show that the remaining term converges to $0$ as $\epsilon\downarrow 0$. But note that
	$$\left|\int_{\Omega\cap B_1}\sum_\alpha\zeta\eta^\prime(v_i)^\alpha_fD_2(v_i)_f^\alpha\right| \leq Q\cdot \sup_\Omega|\zeta|\cdot C\epsilon^{-1}\cdot C\epsilon\cdot \sup_{\Omega\cap\{0<x^2<\epsilon\}}|(v_i)^\alpha_fD_2(v_i)^\alpha_f|$$
	and we know by $(\FB4)$ that $\left.(v_i)_f\right|_{\{x^2=0\}}\equiv 0$, and thus as $v$ is $GC^1$ on $\overline{\Omega}$, we know that $D_2(v_i)_f$ is bounded and $\sup_{\Omega\cap\{0<x^2<\epsilon\}}|(v_i)^\alpha_f|\to 0$ as $\epsilon\downarrow 0$. Thus this term $\to 0$ as $\epsilon\downarrow 0$, which completes the proof.
\end{proof}

\begin{lemma}[$GC^1$ Boundary Squeeze Identity]
	Let $v\in \FB(\BC^{(0)})$ and suppose $v$ is $GC^1$ on some $\overline{\Omega}\subset \overline{H}\cap B_1$, where $\Omega\subset H\cap B_1$ is open. Then, for any $\zeta = (\zeta^2,\dotsc,\zeta^{n+1})\in C^1_c(\overline{\Omega}\cap B_1;\R^n)$ with $\left.\zeta^2\right|_{\del H} \equiv 0$, we have for each $k=1,\dotsc,N$,
	$$\int_{H\cap B_1} \sum_\alpha\sum^n_{i,j=1}\left(|D(v_k)_f|^2\delta_{ij} - 2D_i(v_k)^\alpha_f D_j(v_k)^\alpha_f\right)D_i\zeta^j = 0.$$
\end{lemma}

\begin{proof}
	Let $\eta_\epsilon:\R\to \R$ be a smooth even function obeying $\eta_\epsilon\equiv 0$ on $[0,\epsilon)$, $\eta_\epsilon\equiv 1$ on $[2\epsilon,\infty)$, and $|\eta_\epsilon^\prime|\leq C\epsilon^{-1}$, where $C\in (0,\infty)$ is an absolute constant. Let $\zeta\in C^1_c(\overline{\Omega}\cap B_1;\R^n)$ have $\left.\zeta^2\right|_{\del H}\equiv 0$. Then once again, we know that $\zeta(x)\eta(x^2)$ is a suitable test function for the squeeze identity for any given $v_k$, i.e.
	$$\int_{H\cap B_1}\sum_\alpha\sum_{i,j=1}^n\left(|D(v_k)_f|^2\delta_{ij} - 2D_i(v_k)_f^\alpha D_j(v_k)_f^\alpha\right)\left(\eta_\epsilon(x^2)D_i\zeta^j + \zeta^j \delta_{i2}\eta_\epsilon^\prime(x^2)\right) = 0.$$
	Once again, as $v_k$ is $GC^1$ on $\overline{\Omega}\cap B_1$, we know that as $\epsilon\downarrow 0$, by the dominated convergence theorem,
	\begin{align*}
	\int_{H\cap B_1}\sum_\alpha\sum_{i,j=1}^n&\left(|D(v_k)_f|^2\delta_{ij} - 2D_i(v_k)_f^\alpha D_j(v_k)_f^\alpha\right)\eta_\epsilon(x^2)D_i\zeta^j\\
	&\hspace{3em} \to \int_{H\cap B_1}\sum_\alpha\sum^n_{i,j=1}\left(|D(v_k)_f|^2\delta_{ij} - 2D_i(v_k)_f^\alpha D_j(v_k)_f^\alpha\right)D_i\zeta^j
	\end{align*}
	and thus we just need to show that the second additional term $\to 0$ as $\epsilon\downarrow 0$; this term can be simplified to
	$$\int_{H\cap B_1}\sum_\alpha \left(|D(v_k)_f|^2 - 2|D_2(v_k)_f|^2\right)\zeta^2\eta_\epsilon^\prime(x^2) - 2\int_{H\cap B_1}\eta_\epsilon^\prime\sum_{j>2}\zeta^jD_2(v_k)_f^\alpha D_j(v_k)_f^\alpha.$$
	The first term here can be bounded by:
	\begin{align*}
		&\left|\int_{H\cap B_1}\sum_\alpha (|D(v_k)_f|^2 - 2|D_2(v_k)_f|^2)\zeta^2\eta^\prime_\epsilon(x^2)\right|\\
		& \hspace{3em} \leq C\epsilon^{-1}\cdot C\epsilon\cdot \sup_{\Omega\cap \{0<x^2<\epsilon\}}|\zeta^2|\cdot \sup_{\Omega}\left||D(v_k)_f|^2 - 2|D_2(v_k)_f|^2\right|
	\end{align*}
	which goes to zero as $\epsilon\to 0$, as $\left.\zeta^2\right|_{\{x^2=0\}} \equiv 0$ and $v_k$ is $GC^1$ on $\overline{\Omega}$, so the second supremum is finite. The last term is
	$$\left|\int_{H\cap B_1}\eta_\epsilon^\prime\sum_{j>2}\zeta^jD_2(v_k)^\alpha_fD_j(v_k)^\alpha_f\right| \leq C\epsilon^{-1}\cdot C\epsilon \cdot \sup_{\Omega}|\zeta^jD_2(v_k)^\alpha_f|\cdot \max_{j>2}\left(\sup_{\Omega\cap\{0<x^2<\epsilon\}}|D_j(v_k)^\alpha_f|\right)$$
	and note that as $v_k$ is $GC^1$ on $\overline{\Omega}$ we know that the $\sup_{\Omega}|\zeta^jD_2(v_k)^\alpha_f|$ is finite, and moreover by $(\FB4)$ we know that $\left. (v_k)_f\right|_{\{x^2=0\}}\equiv 0$, and so $\left.D_j(v_k)_f^\alpha\right|_{\{x^2=0\}}\equiv 0$ for all $j>2$, and thus as $v_k$ is $GC^1$ on $\Omega$, we have $\sup_{\Omega\cap \{0<x^2<\epsilon\}}|D_j(v_k)_f^\alpha| \to 0$ as $\epsilon\to 0$; hence this completes the proof.
\end{proof}

\textbf{Remark:} The above proofs show that we do not need any further boundary estimates (once one has those in the interior) to prove that one can classify the homogeneous degree one blow-ups in this non-flat setting (and hence prove a $GC^{1,\alpha}$-regularity result up-to-the-boundary), as long as one can show that property $(\FB4)$ holds, i.e. density gaps do not occur (or at least can deal with density gaps in an alternative fashion to prove the squeeze identity). In particular, we do not need to prove an energy non-concentration estimate (as was done in \cite{minterwick}) at the boundary. \hfill $\blacktriangle$

\appendix

\section{Boundary Regularity of Certain Dirichlet-Stationary Functions}

In this appendix we note how the proof technique used in Section \ref{sec:blow-up} and Section \ref{sec:squash-squeeze} can be used to generalise the boundary regularity results seen in \cite[Section 3]{minter-campanato} to a larger class of functions which instead of being multi-valued harmonic (in a suitable sense) in the interior, we instead assume they are stationary for the Dirichlet energy in the following sense:

\begin{defn}
	Fix $\Omega\subset\R^n$ an open subset and $q\in \Z_{\geq 1}$. We say $v\in W^{1,2}(\Omega;\A_q(\R^k))$ is \textit{Dirichlet-stationary} if it is stationary for the Dirichlet energy under domain deformations and ambient deformations, namely if both the following conditions hold:
	\begin{enumerate}
		\item [(IV)] Given $\phi\in C^\infty_c(\Omega;\R^n)$, for $\epsilon>0$ sufficiently small, the map $\Phi_\epsilon:\Omega\to \R^n$ given by $\Phi_\epsilon(x) := x+\epsilon\phi(x)$ is a diffeomorphism of $\Omega$ which fixes $\del\Omega$; we assume that for any such $\phi$,
		$$\left.\frac{\ext}{\ext\epsilon}\right|_{\epsilon=0}\int_\Omega |D(v\circ\Phi_\epsilon)|^2 = 0;$$
		\item [(OV)] Given $\psi\in C^\infty(\Omega\times\R^k;\R^k)$ such that $\spt(\psi)\subset\Omega^\prime\times\R^k$ for some $\Omega^\prime\subset\subset\Omega$, if we set $\Psi_\epsilon(x):= \sum_{i=1}^Q\llbracket v_i(x) + \epsilon \psi(x,v_i(x))\rrbracket$, then we assume
		$$\left.\frac{\ext}{\ext\epsilon}\right|_{\epsilon=0}\int_{\Omega}|D\Psi_\epsilon|^2 = 0.$$
	\end{enumerate}
\end{defn}

Here, ``IV'' stands for \textit{inner variation} and ``OV'' stands for \textit{outer variation}. For $v = \sum_{\alpha=1}^Q\llbracket v^\alpha\rrbracket \equiv \sum_\alpha\llbracket (v^\alpha_1,\dotsc,v^\alpha_k)\rrbracket$ Dirichlet-stationary, one may then derive two variational identities (see \cite[Proposition 3.1]{de2010almgren}) from (IV) and (OV), namely the \textit{squeeze} and \textit{squash} identities we have already seen in Section \ref{sec:squash-squeeze}:
\begin{equation}\label{squeeze}
	\int_\Omega \sum_{\alpha=1}^Q\sum_{i,j=1}^n\left(|Dv^\alpha|^2\delta_{ij} - 2D_i v^\alpha D_j v^\alpha\right)D_i\phi^j = 0 \ \ \ \ \text{for any }\phi\in C^\infty_c(\Omega;\R^n);
\end{equation}
and
\begin{equation}\label{squash}
	\int_\Omega \sum_{\alpha=1}^Q\sum_{i=1}^n\sum^k_{p=1}|D_iv_p^\alpha|^2\phi + \int_{\Omega}\sum_{\alpha=1}^Q\sum_{p=1}^k v^\alpha_p Dv^\alpha_p\cdot D\phi = 0\ \ \ \ \text{for any }\phi\in C^\infty_c(\Omega;\R);
\end{equation}
note that (\ref{squeeze}) is exactly \cite[(3.3)]{de2010almgren}, whilst (\ref{squash}) follows from \cite[(3.5)]{de2010almgren} by taking $\psi(x,u) = \phi(x)u$, for $\phi\in C^\infty_c(\Omega;\R)$.

If in Definition \ref{def:blow-up}, instead of $(\FB2)$ one assumes that each $v^i$ is Dirichlet-stationary and $GC^1$ in $H\cap B_1$ and that any $v$ which is homogeneous of degree one and Dirichlet-stationary on all of $\R^k\setminus\{0\}$ with $v^i$ having no points of density $q_i$ (i.e. where all values of $v^i$ coincide) except at $0$ is necessarily a union of $q_i$ linear functions, then one may establish the same classification of homogeneous degree one blow-ups as in Lemma \ref{lemma:classification}. Thus, if one instead assumes in $(\FB2)$ that each $v^i$ is $GC^{1,\alpha}$ in $H\cap B_1$ for some $\alpha>0$, then one can establish a variant of Theorem \ref{thm:boundary-reg} in this setting; we remark that such a regularity assumption in the interior would come from, for example, a frequency lower bound at the branch set. The upshot is, after establishing a suitable regularity theorem near hyperplanes of certain multiplicities (say, $GC^{1,\alpha}$ for some fixed $\alpha>0$ independent of the blow-up) such that the blow-ups obey both the squash and squeeze inequality in the interior (i.e. are essentially Dirichlet-stationary in the above sense), proving a regularity result for such varifolds near classical cones with half-hyperplanes having the same range of multiplicities essentially just relies on knowing that density gaps do not occur if one follows the general technique outlined in the current work, which is based on \cite{minter-5-2}, \cite{minter-campanato}, and \cite{minterwick}.

\bibliographystyle{alpha} 
\bibliography{references}

\end{document}